\documentclass[11pt]{article}
\usepackage{ifthen}
\newboolean{arxiv}
\setboolean{arxiv}{true}

\usepackage{ifthen}
\usepackage{mathtools}
\mathtoolsset{showonlyrefs}

\ifthenelse{\boolean{arxiv}}{
\usepackage{arxiv_def}
}
{
\usepackage{times}
\usepackage{colt_def}
}

\usepackage[toc,page]{appendix}
\usepackage{xr}
\usepackage{mathtools}
\usepackage{yfonts}

\DeclareSymbolFont{Greekletters}{OT1}{iwona}{m}{n}
\DeclareSymbolFont{greekletters}{OML}{iwona}{m}{it}
\DeclareMathSymbol{\salpha}{\mathord}{greekletters}{"0B}
\DeclareMathSymbol{\sbeta}{\mathord}{greekletters}{"0C}
\DeclareMathSymbol{\sgamma}{\mathord}{greekletters}{"0D}
\DeclareMathSymbol{\sOmega}{\mathord}{Greekletters}{"0A}
\DeclareMathSymbol{\smu}{\mathord}{greekletters}{"16}
\DeclareMathSymbol{\svarepsilon}{\mathord}{greekletters}{"22}
\DeclareMathSymbol{\svarrho}{\mathord}{greekletters}{"25}
\DeclareMathSymbol{\svarphi}{\mathord}{greekletters}{"27}

\let\De\cuD

\makeatletter
\newcommand{\vast}{\bBigg@{3}}
\newcommand{\Vast}{\bBigg@{4}}
\makeatother

\begin{document}

\title{Variational Formulas for the Spectrum of\\ Block Wishart Matrices}
\author{
Andrea Montanari\thanks{Department of Statistics and Department of Mathematics, Stanford University} 
	\and 
	Basil Saeed\thanks{Department of Electrical Engineering, Stanford University}
}

\maketitle

\begin{abstract}
We analyze the asymptotics of a block-Wishart random matrix ensemble of the type $\bW_k = (\bX^* \otimes \bI_k)\bT (\bX\otimes\bI_k)$
for $\bX \in\C^{n\times p}$ with i.i.d. rows
satisfying a suitable concentration-of-measure property,
and $\bT := \Diag(\bT_i)_{i\in[n]}$ a block diagonal matrix with self-adjoint blocks $\bT_i\in \C^{k\times k}$,
under the proportional asymptotics $n/p\to\alpha$
with $k$ fixed. These matrices play a prominent
role in the analysis of $k$-index models in high-dimensional statistics.
By studying the matrix Stieltjes transform of this random matrix model 
and its inverse ($K$-transform), we derive variational formulas for two functionals of the asymptotic spectral density of $\bW_k$: the left (equivalently right) edge of its support, and its logarithmic potential.
\end{abstract}

\section{Introduction}
Let $\bX_n \equiv\bX \in \C^{n\times p}$ be a sequence of random matrices indexed by $n \in \N$, 
with i.i.d. rows $\bx_i \in\C^p, i\in[n]$, and for each $n$, consider matrices $\{\bT^{(n)}_1,\dots,\bT^{(n)}_n\}$ 
that are deterministic or independent of $\bX$,
with  $\bT_i\in\sfM_{k}(\C)_{\textrm{sa}}$ for each $i\in[n]$.
(Here, we denote by $\sfM_{k}(\C)_{\textrm{sa}}$ the set of self-adjoint $k\times k$ matrix with elements in $\C$.)
In order to lighten the notation, we will drop the superscript from $\bT^{(n)}_i$,

We will be interested in the asymptotics of Wishart-type matrices defined by
\begin{equation}
\label{eq:W_k_def}
    \bW_k := \frac1n(\bX^* \otimes \bI_{k})\Diag(\bT_{i})_{i\in[n]} (\bX\otimes \bI_k)
    \equiv \frac1n\sum_{i=1}^n (\bI_k\otimes \bx_i) \bT_i(\bI_k\otimes \bx_i^*)
    \in\C^{pk \times pk}
\end{equation}
in the proportional regime 
\begin{equation}
    \frac{n}{p} \to \alpha\in(0,\infty) \quad\quad n,p\to\infty,
\end{equation}
when
\begin{equation}
    \frac1n \sum_{i=1}^n \delta_{\bT_i} \Rightarrow {\nu} \in\cuP_0(\sa{k}{\C}),
\end{equation} 
here and throughout,
$\cuP(\cX)$ denotes the collection of probability measures on the Polish space $\cX$, while
$\cuP_0(\cX)$ denotes those with compact support.
We will assume throughout the $\bx_i$ to be centered (i.e. $\E[\bx_i]=\bzero$)
and isotropic (i.e. $\E[\bx_i\bx_i^{*}]=\bI_p$).

Matrices of this type appear in several applications, for example, in communications $\bW_k$ is the Gram matrix for MIMO channels describing the effect of the channel on the 
signal~\cite{couillet2011random}. In high-dimensional statistics and machine learning,
the matrix $\bW_k$ is the Hessian of a general 
empirical risk $\hat R_n(\bTheta) := n^{-1}\sum_{i\le n} \ell(\bTheta^{\sT}\bx_i)$ 
where $\bTheta \in\R^{p\times k}$ is a matrix of vector parameters \cite{asgari2025local,arous2025local}.
In this setting, parameters are typically estimated by minimizing $\hR_n(\bTheta)$, and it is therefore
of interest to study the properties of the Hessian at a point $\bTheta$:
\begin{align}
    \nabla^2\hR(\bTheta) = \frac{1}{n}\sum_{i=1}^n
    (\bI_k\otimes \bx_i) \nabla^2
    \ell(\bTheta^{\sT}\bx_i)(\bI_k\otimes \bx_i^*)\, .
\end{align}
This is the same as Eq.~\eqref{eq:W_k_def}
where we identify $\bT_i=   \nabla^2\ell(\bTheta^{\sT}\bx_i)$.

Two remarks are in order with regard to the application to
empirical risk minimization. \emph{First,}
unlike in our setting  $\nabla^2\ell(\bTheta^{\sT}\bx_i)$
is not independent of $\bx_i$. However this dependence can often be accounted for by a low rank deformation
\cite{mondelli2019fundamental,lu2020phase}.
\emph{Second,} when the cost function $\ell$ is non-convex, the matrices $\bT_i$ are not necessarily positive semidefinite (PSD). Moving beyond the PSD case will be a 
main motivation for the present note.

For $k=1$, the ensemble $\bW_1$ is
the classical Wishart-type matrix deformed multiplicatively by the diagonal entries $t_i$, i.e.,
\begin{equation}
\label{eq:wishart_classical}
    \bW_1 = \frac1n \sum_{i=1}^n  t_i\bx_i \bx_i^* \in \R^{p\times p}
\end{equation}
with $\bx_i \in\R^p$ the rows of $\bX$. Under certain tail decay conditions on $\bx_i$, the asymptotic spectral density in this setting was characterized  by the seminal work~\cite{marvcenko1967distribution,silverstein1995empirical} (with several refinements since then, e.g.~\cite{hachem2007deterministic})
as the measure $\mu_\up{\alpha,\nu}$ whose Stieltjes transform $s_\up{\alpha,\nu}(z)$ uniquely solves
\begin{equation}
\label{eq:ST_FP_EQ_1}
   K_{\up{\nu,\alpha}}(s) = z \quad\quad z \in \bbH_+, \quad\quad K_\up{\nu,\alpha}(s) := \int \frac{t}{1 + s t} \nu(\de t) - \frac1{\alpha s} 
\end{equation}
where $\bbH_+ := \{z\in \C : \Im(z) > 0\}$. Following \cite{lehner1999computing},
we will refer to $K_{\up{\nu,\alpha}}(s)$ (and its generalizations below)
as the \emph{$K$-transform}\footnote{The term ``$K$-transform'' is typically reserved for the inverse of the \emph{Cauchy transform}, but we will use it here for the inverse of the Stieltjes transform.}.  

For general $k>1$, the asymptotics of $\bW_k$ are given by the following proposition, which is a direct generalization of 
Eq.~\eqref{eq:ST_FP_EQ_1} and follows 
from  earlier work (see Remark \ref{rmk:PropoProof}
below).
We note here that this asymptotic characterization holds 
under various other assumptions on the matrix $\bX$.
For concreteness, we state it here under the assumption that the rows of $\bX$ enjoy 
a Hanson--Wright type concentration inequality, but this is not the only setting.

\begin{proposition}[Standard asymptotic characterization]
\label{prop:asymptotic_FP}
Assume the rows $\bx_i\in\C^{p}$ are i.i.d., centered (i.e., $\E[\bx_i] =\bzero$), and isotropic
(i.e., $\E[\bx_i\bx_i^*] = \bI_p$). Assume further that there exists 
$\sfK \in (0,\infty)$
such that, for each $i\in[n]$, all deterministic $\bA\in \C^{p\times p}$, and all $t>0$,
\begin{equation}
\label{eq:HW_condition}
   \P\Big( \big|\bx_i^*\bA \bx_i - \Tr(\bA)\big| \ge t \Big)
   \le
   2 \exp\bigg\{-c \min\bigg(\frac{t^2}{\sfK^4 \|\bA\|_\Fnorm^2},\, 
   \frac{t}{\sfK^2 \|\bA\|_\op}\bigg)\bigg\}\, ,
\end{equation}
for some universal constant $c>0$, where $\|\bA\|_\Fnorm$ and $\|\bA\|_\op$ denote 
the Frobenius and operator norms of $\bA$, respectively.
Fix integer $k\ge 1$, $\alpha \in (0,\infty)$ and
let $\nu\in\cuP_0(\sa{k}{\C})$.
Define
\begin{equation}
\label{eq:K_transform}
    \bK_\up{\alpha,\nu}(\bS) := \int  \bT(\bI + \bS\bT)^{-1} \nu(\de \bT) - \frac1\alpha\bS^{-1},\quad\quad \bK_\up{\alpha,\nu}:\C^{k\times k} \to\C^{k\times k}.
\end{equation}

For any $z \in\bbH_+$, there exists a unique 
solution 
$\bS_\up{\nu,\alpha}(z)$
to
\begin{equation}
\label{eq:ST_FP}
    \bK_\up{\alpha,\nu}(\bS) = z\bI,\quad\quad \bS\in\bbH_+^k := \{\bB\in \C^{k\times k}: \Im(\bB) \succ\bzero\}.
\end{equation}
Furthermore,
    \begin{equation}
        s_\up{\nu,\alpha}(z) := 
         \frac{\alpha}{k} \Tr(\bS_\up{\alpha,\nu}(z))\quad\quad\textrm{for}\quad z\in\bbH_+
    \end{equation}
is the Stieltjes transform of a probability measure 
$\mu_\up{\alpha,\nu}\in\cuP(\R)$ such that, almost surely
\begin{equation}
    \hmu_{\bW_k} \Rightarrow \mu_\up{\alpha,\nu} \quad\quad \textrm{almost surely, as} \quad\quad n\to\infty,\;\frac{n}{p} \to\alpha,
\end{equation}
where $\hmu_{\bW_k}$ is the empirical spectral distribution of $\bW_k$ and $\Rightarrow$ denotes weak convergence of probability measures.

Finally, if $\dist(\{\bT_i:i\le n\},\supp(\nu))\to 0$ (with $\dist$ denoting
Hausdorff distance between sets in $(\reals^k,\|\,\cdot\,\|$), then,
almost surely,
\begin{align}
\dist\big(\supp(\hmu_{\bW_k}),\supp (\mu_\up{\alpha,\nu})\big) \to 0\, .
\end{align}
\end{proposition}

\begin{remark}\label{rmk:PropoProof}
The proof of Proposition \ref{prop:asymptotic_FP} follows from several standard results 
in the literature. One approach is to use the linearization trick (see~\cite{far2006spectra,speicher2020lecture}), noting that matrix $\bW_k$
takes the form (for a suitable ordering of rows and columns)
\begin{align}
\bW_k = \left[\begin{matrix}
\bU_{11} &\bU_{12} &\cdots & \bU_{1k}\\
\bU_{21} & \bU_{22} & \cdots & \bU_{2k}\\
\vdots & \vdots & \cdots & \vdots\\
\bU_{k1} &\bU_{k2} &\cdots & \bU_{kk}\\
\end{matrix}\right]\, ,
\;\;\;\;\;\; \bU_{ij}:=\frac{1}{n}\bX^{*}\hat{\bT}_{ij}\bX\, ,
\end{align}
where $\hat{\bT}_{ij}:=\Diag((\bT_{\ell})_{ij}:\ell\le n)$. 
Hence, for any integer $m$, there is a non-commutative polynomial  $Q_m$ such that
$\Tr(\bW^{m}_k) = \Tr(Q_m(\{\bU_{ij}\}))$, and the asymptotics of the latter
for Gaussian matrices
follows from \cite[Theorem 5.4.10]{anderson2010introduction} (for weak convergence
of the empirical spectral distribution)
and \cite[Theorem 5.5.1]{anderson2010introduction} 
(for convergence of the support).
For non-Gaussian  matrices, the same asymptotics holds by a Lindeberg argument
\cite{chatterjee2006generalization,montanari2022universality}.
An alternative approach is to use the Matrix Dyson equation of
\cite{alt2017local}.

For a complete proof of the exact statement of Proposition~\ref{prop:asymptotic_FP} in the Gaussian setting (in particular, existence and uniquenness of
the solution of Eq.~\eqref{eq:ST_FP}, see the Appendix of \cite{asgari2025local}.
It is not too hard to see that their proof is directly generalizable to the current setting.
\end{remark}

\section{Computing functionals of $\mu_\up{\alpha,\nu}$}
While the above proposition characterizes the asymptotic spectral measure $\mu_\up{\alpha,\nu}$ of $\bW_k$, it does so in an implicit way and  studying functionals of this measure from this equation can be  challenging.
We will focus here on two particular functionals of the measure
$\mu_\up{\alpha,\nu}$: the \emph{edges of its support},
\begin{equation}
    \zeta_{-,\up{\alpha,\nu}} := \inf\supp(\mu_{\up{\alpha,\nu}}),\quad\quad
    \zeta_{+,\up{\alpha,\nu}} := \sup\supp(\mu_{\up{\alpha,\nu}}),
\end{equation}
and its \emph{logarithmic potential}
\begin{equation}
    \Phi_\up{\alpha,\nu}(u) := \int \log |u -\zeta| \mu_\up{\alpha,\nu}(\de \zeta),\quad\quad u \in \R \setminus [\zeta_-,\zeta_+].
\end{equation}
Both of these quantities are useful in applications.
The logarithmic potential coincides with the limit of 
$\log\det(\bW_k)$, which 
in wireless communications,
determines \emph{channel capacity} (i.e., the mutual information) of a MIMO channel modeled by $\bW_k$.

In high-dimensional statistics, $\zeta_{-}$ can be used
to compute the lower edge of the asymptotic
spectrum of the Hessian $\nabla^2\hat R_n(\bTheta)$. 
Hence, it provides information about the stability of landscape at a particular point $\bTheta$. The log potential also plays an important role in the Kac-Rice approach
to the analysis of the empirical risk minimization problem 
\cite{maillard2020landscape,asgari2025local}.

In the rest of this section, we will present the main ideas for the functional $\zeta_-$ 
to be definite, and state the characterization of the log-potential at the end of the section.
In what follows, we will often omit the subscript $\up{\alpha,\nu}$ from all quantities whenever there is no risk of confusion; for example, we write $\mu$ for $\mu_\up{\alpha,\nu}$.

In the classical setting $k=1$, \cite{silverstein1995analysis} 
already gives a useful characterization of $\zeta_-$.
Namely, letting $t_{\min} := \inf \supp(\nu) \wedge 0$, and recalling the 
definitions  of~\eqref{eq:ST_FP_EQ_1}, we have
\begin{equation}
    \zeta_{-}= \sup_{s\in \cM_1} K(s)\quad\quad\textrm{with}\quad\quad
    \cM_1 := \left\{s \in (0, -t_{\min}^{-1}),\; \int \left(\frac{s t }{1 + st} \right)^{2} \nu(\de t)< \frac1\alpha\right\}.\label{eq:OnedCase}
\end{equation}
(Here, it is understood that 
$-1/t_{\min}=\infty$ whenever $t_{\min}=0$.)
Namely, if $x < \inf \supp(\mu),$ then $s(x) \in (0,-t_{\min}^{-1})$
where $t_{\min} := \inf \supp(\nu) \wedge 0.$ 
This formula can be read off of Eq.~\eqref{eq:ST_FP_EQ_1}  as follows:
By definition, $s(z)$ has an analytic continuation on a the maximal interval $(-\infty,\zeta_{-})$ (i.e. no analytic continuation exists on $(-\infty,\zeta_{-}+\eps)$
for any $\eps>0$). Further $s$ maps $(-\infty,x)$ to $(0,s(x))$,
$s(x)>0$ for every $x<\zeta_-$.
Since $K(s)$ is its analytic inverse on $(0,s(x))$, then 
we must have $\zeta_{-} = K(s_0)$  when 
$s_0 >0$ is the smallest value at which $K$ fails to be invertible. 
This coincides with the smallest $s$ at which $K'(s)\le 0$, which 
yields Eq.~\eqref{eq:OnedCase}, by noting that
\begin{align}
s^2K'(s) =  \frac{1}{\alpha}-\int \left(\frac{s t }{1 + st} \right)^{2} \nu(\de t)\, .
\end{align}
In other words,
\begin{align}
\label{eq:pre_var_zmin_1}
 s_\up{\alpha,\nu}( (-\infty, \zeta_{-})) = \cM_1 ,
\end{align}
which is equivalent to Eq.~\eqref{eq:OnedCase}.

A naive attempt to generalize this idea to
$k>1$ yields
\begin{equation}
\label{eq:pre_var_zmin_k}
   \zeta_{-}  = \sup \big\{z : \exists \bS\in \cM_k \;\;\textrm{s.t.}\;\; \bK(\bS)=z\bI_k  \big\}
\end{equation}
with $\bK$ the $K$-transform defined in Proposition~\ref{prop:asymptotic_FP},
and the set $\cM_k$ defined as the (bijective) image of $(-\infty,\zeta_{\min})$
under $\bS_\up{\alpha,\nu}(x)$.
However, this description of the set $\cM_k$ is not explicit, and the resulting 
characterization is circular.

In order to obtain a more useful generalization of Eq.~\eqref{eq:OnedCase}, we will consider  
the matrix valued $K$-transform of Eq.~\eqref{eq:K_transform}
$\bK:\sfM_k(\C)\to\sfM_k(\C)$, and in the equation $\bK(\bS) = \bZ$ for general $\bZ\in\C^k$.
As we will see below, the inverse of $\bS\mapsto\bK_{\up{\alpha,\nu}}(\bS)$,
call it $\hat\bS_{\up{\alpha,\nu}}(\bZ)$ has the interpretation of providing the limit
\begin{align}
    \hat\bS_{\up{\alpha,\nu}}(\bZ)= \lim_{n,p\to\infty}\frac{1}{n}\Big(\bI_k\otimes \Tr\Big)\Big((\bW_k-\bZ\otimes \bI_p)^{-1}\Big)\,
\end{align}
for more general $\bZ\in\C^k$ than multiples of the identity.
The details of this approach will be clarified in Section~\ref{sec:proofs_main}.

We can now state our main results. We denote by
$\cuD \bK(\bS)[\cdot]:\sfM_{k}(\C)\to\sfM_{k}(\C)$ the 
directional derivative of $\bK$ at $\bS$,
which can be computed explicitly as
\begin{equation}
   \De_\bS\bK(\bS)[\bDelta] =
   \frac1\alpha \bS^{-1}\bDelta\bS^{-1} - 
   \int \bT(\bI+\bS\bT)^{-1}\bDelta (\bI+\bT\bS)^{-1}\bT\, \nu(\de\bT) \quad\quad
\textrm{for}\quad
 \bDelta\in\sa{k}{\C}.
\end{equation}
Notice that  $\cuD \bK(\bS)[\cdot]$ is a self-adjoint operator (with respect to the standard 
scalar product on $\sfM_{k}(\C)$). Hence it 
makes sense to ask whether $\De_\bS\bK(\bS)[\cdot ]$ is positive definite 
(in formulas $\De_\bS\bK(\bS)[\cdot ]\succ 0$) or positive semidefinite 
($\De_\bS\bK(\bS)[\cdot ]\succeq 0$),
which amounts respectively to  $\<\bDelta,\De_\bS\bK(\bS)[\bDelta ]\> > 0$ for all $\bDelta\neq \bzero$,
or  $\<\bDelta,\De_\bS\bK(\bS)[\bDelta ]\> \ge 0$ for all $\bDelta\neq \bzero$.
\begin{theorem}[Variational formula for left edge]
\label{thm:edge}
Fix $k \ge 1$ integer, $\alpha \in (0,\infty)$, and $\nu\in\cuP_0(\sa{k}{\C})$.
Let $\mu_{\alpha,\nu}$ 
be as in Proposition~\ref{prop:asymptotic_FP}.
Then we have
\begin{equation}
 \inf \supp(\mu_\up{\alpha,\nu})= \sup_{\bS \in \cP_{-}} \lambda_{\min}\left(\bK_{(\alpha,\nu)}(\bS)\right). 
\end{equation}
where $\cP_-$ is given by
   \begin{equation}
   \label{eq:P_def}
  \cP_- \equiv \cP_-(\alpha,\nu) := \{\bS\succ\bzero : \bS^{-1}+ \bT\succ\bzero \;\; \forall\,\bT\in\supp(\nu),\;\; 
  \cuD_\bS \bK(\bS)[\,\cdot\,]\succ 0\; 
  \}.
   \end{equation}
   Furthermore, the set $\cP_-$ is equivalent to
   \begin{equation}
   \label{eq:P_def_with_B}
  \cP_- = \{\bS\succ\bzero : \bS^{-1}+ \bT\succ\bzero \;\; \forall\,\bT\in\supp(\nu),\;\;\textrm{there exists  }\bB\succ\bzero \textrm{ s.t.}  \;
  \cuD_\bS \bK(\bS)[\bB]\succ \bzero\; 
  \}.
   \end{equation}
\end{theorem}
A related characterization for the left and right edges $\zeta_{\pm}$ appeared recently in~\cite{parmaksiz2025computing} applicable when $\bT_i$ are positive semi-definite, 
albeit stated differently. The derivation of \cite{parmaksiz2025computing}  
is based on a technique introduced in~\cite{lehner1999computing} 
to compute the operator norm of free operators. 
The approach of \cite{lehner1999computing}  realizes the edge 
as the operator norm of a free operator which can be explicitly decomposed as the sum of 
operators on the Fock Space. By carefully bounding the norm of these basis operators, they obtain a tight upper bound on the operator norm\footnote{After 
finishing this manuscript, Ramon van Handel communicated to us that the
technique of \cite{parmaksiz2025computing}  can also be generalized to indefinite
case, yielding a characterization analogous (but identical) to 
Theorem \ref{thm:edge}}.  We point out that the setting of
\cite{parmaksiz2025computing} is more general in other ways, and in particular accommodating for non-isotropic\footnote{We believe our approach can be generalized to the non-isotropic case but we refrain from doing so at the moment.} $\bx_i$'s.

The next theorem gives the formula for logarithmic potential $\Phi(u)$.
\begin{theorem}[Variational formula for the logarithmic potential]
\label{thm:log}
Fix  $k \ge 1$ integer,
$\alpha \in (0,\infty)$, and $\nu\in\cuP_0(\sa{k}{\C})$.
For $u<\zeta_{-,\up{\alpha,\nu}}$,
set
\begin{equation}
    g_\up{\alpha,\nu}(\bS;u) := -\alpha u \Tr(\bS) + \alpha \int \log \det(\bI + \bT \bS) \nu(\de\bT) - \log|\det(\bS)| - k(\log \alpha +1 ).
\end{equation}
then we have (for $\cP_-$ defined as in Eq~\eqref{eq:P_def})
    \begin{equation}
  \Phi(u) := \int \log |\zeta - u| \, \mu_{\up{\alpha,\nu}}(\de \zeta) 
    =  \frac1k\inf_{\bS\in\cP_{-}} g_\up{\alpha,\nu}(\bS;u)\, .
    \end{equation}
\end{theorem}

We remark that since $\bT$ is not assumed to be positive semi-definite, the above two theorems directly lead to a formula for $\zeta_+$, and for $\Phi(u)$ for $u> \zeta_+$, respectively. We state the exact form in Remarks~\ref{remark:zeta_+} and~\ref{remark:log_+} below.

\section{Proofs of Theorem \ref{thm:edge} and \ref{thm:log}}
\label{sec:proofs_main}
In this section, we prove Theorems~\ref{thm:edge} and~\ref{thm:log}
by leveraging properties of the matrix Stieltjes and $K$-transforms.
We defer the proof of the key technical fact, Theorem 
\ref{thm:K_transform_real}, to the next section.

We begin with some background on free probability. Most of these preliminaries are taken from~\cite{nica2006lectures} and~\cite{speicher2020lecture}.

\paragraph{Free probability preliminaries.}
Let $(\cA,\tau)$ be a $C^\star$-probability space (i.e., it is both tracial and faithful), with elements $m,t_{i,j} \in \cA$, $i,j\in[k]$
satisfying the following conditions. First,
$m$ is a \emph{free-Poisson} element with rate $\alpha$, i.e., 
for any integer $d>0$,
\begin{equation}
\tau(x^d) = \int \zeta^d\, \mu_{\MP,\alpha}(\de \zeta)
\end{equation}
with $\mu_{\MP,\alpha}$ denoting the Marchenko-Pastur law with aspect ratio $\alpha$.
Second, $t_{i,j}$ for $i,j\in[k]$ satisfy, for any pairs $Q = \{(a_1,b_1),\dots,(a_N,b_N)\}$, $a_i,b_i \in [k],$
\begin{equation}
    \tau\left(\prod_{(i,j) \in Q} t_{i,j}\right) = \int \prod_{(i,j)\in Q} (\bT)_{ij} \;  \nu(\de \bT).
\end{equation}
Define
\begin{align}
    w_k := ( m^{1/2} t_{i,j} m^{1/2})_{i,j \in[k]} \in \sa{k}{\cA}.\label{eq:wk-def}
\end{align}
Here, we use $\sa{k}{\cA}$ to denote the set $k$ by $k$ self-adjoint matrices with elements in the $C^*$-algebra $\cA$.
Since $\nu\in\cuP(\sa{k}{\C}),$ it is easy to see that $w_k$ will indeed be self-adjoint. The operator $w_k$ can be seen as the \emph{free-analogue} or \emph{free deterministic equivalent} of $\bW_k$ from Eq.~\eqref{eq:W_k_def} in the following sense: for any non-commutative polynomial $f$ of $N>0$ elements, and pairs 
$Q = \{(a_1,b_1),\dots,(a_N,b_N)\}$, $a_i,b_i \in [k],$
\begin{equation}
    \label{eq:free_determinism}
    \lim_{\substack{n\to\infty \\ n/p \to\alpha}}\E\left[\frac1n 
   \Tr f\left( \frac1n \bX^*\hat\bT_{a_1b_1}\bX,\dots,  \frac1n \bX^*\hat\bT_{a_Nb_N}\bX \right)
    \right]
    =
    \frac1\alpha\tau\left(
    f\left( m^{1/2}t_{a_1b_1}m^{1/2},\dots,  m^{1/2}t_{a_Nb_N}m^{1/2}\right)
    \right)
\end{equation}
where $\hat\bT_{a,b}$ is the $n$ by $n$ matrix whose diagonal elements are given by $(\bT_{1})_{a,b}, \dots, (\bT_{n})_{a,b}$.

Given an element $a\in\cA$, its spectrum can be defined as the support of the measure $\mu$ on
$\R$ whose moments 
coincide with the the moments of $a$. We denote this by $\spec_\cA(a).$
The Gelfand-Naimark-Segal construction guarantees the existence of a Hilbert space $\cH$, a $*$-representation $\Pi$ of $\cA$ on $\cH$, and 
and a `tracial' element $\psi\in\cH$ such that for
any $a\in\cA,$  
\begin{equation}
    \tau(a) = \langle \psi, \Pi(a) \psi\rangle_{\cH}.
\end{equation}
The $*$-representation $\Pi$ satisfies $\Pi(a^*) = \Pi(a)^*$,
where the $A^*$ denotes the adjoint of operator $A$. 
In particular, if $a\in\cA$ is self-adjoint, then the spectral theory of bounded operators on $\cH$ defines the spectrum of $\Pi(a)$ which we denote by $\spec_\cH(\Pi(a)).$
Faithfulness of $\tau$ then guarantees the two notions of the spectrum of $a$ coincide.

\begin{proposition}[Proposition 3.15 of~\cite{nica2006lectures}] 
\label{prop:spectra_match}
For $(\cA,\tau)$ as above with $\tau$ faithful,
if $a\in\cA$ is self-adjoint, then
\begin{equation}
    \spec_\cA(a)= \spec_{\cH}(\Pi(a)).
\end{equation}
\end{proposition}
Having this in mind, we will throughout use $\spec(a)$ to denote both $\spec_\cH(\Pi(a))$ and $\spec_\cA(a)$.

Define the product Hilbert space $\cH^k := \cH\times \cdots \times \cH$ ($k$ times Cartesian product).
Equivalently $\cH^k=\C^k\otimes \cH$.
Then note that
$(\bI_k\otimes \Pi)w_k$ can be seen as a element of $\cB(\cH^k).$
For a bounded self-adjoint operator $B$ on $\cH$ or $\cH^k$, we write
$B\succ 0$ if $\spec(B)\subseteq (0,\infty)$. 
For a self-adjoint $a\in \cA$ (or, $a_k\in\sa{k}{\cA}$),
we will use $a\succ\bzero$  ($a_k \succ\bzero$) instead of the more cumbersome 
$\Pi(a)\succ\bzero$ ($(\bI_k\otimes\Pi)a\succ\bzero$, respectively).
Finally, define the matrix Stieltjes transform by
\begin{equation}
\label{eq:free_ST_def}
\hat\bS_\up{\alpha,\nu}(\bZ) := \frac1\alpha (\bI_k \otimes \tau)\Big( 
w_k - \bZ\otimes \id_\cA  \Big)^{-1},\quad\quad
\bZ\in\bbH_+^k,
\end{equation}
where $\bbH_+^k$ was defined in Eq.~\eqref{eq:ST_FP}.
Note that under the conditions of Proposition~\ref{prop:asymptotic_FP},
this quantity is the almost sure limit of the random matrix
\begin{equation}
\label{eq:empirical_ST}
    \hat\bS_n(\bZ) := 
    (\bI_k \otimes \frac1n \Tr) \left( \bW_k - (\bZ \otimes \bI_p)\right)^{-1}
\end{equation}
as $n,p\to\infty$ with $n/p\to\alpha.$ 

\paragraph{The image of $\cU_-$.}
The following theorem gives the inverse of $\hat\bS$, to the `left' and `right' of the spectrum of $w_k$.
This is the key technical result to establish the characterizations of $\zeta_{\pm}$
and $\Phi(u)$
from Theorems~\ref{thm:edge} and~\ref{thm:log}.

\begin{theorem}
\label{thm:K_transform_real}
Fix integer $k \ge 1,$
let $\alpha \in (0,\infty)$, and $\nu\in\cuP_0(\sa{k}{\C})$. Define   
$\cU_-\equiv \cU_{-}(\alpha,\nu)\subseteq \sa{k}{\C} $ via
\begin{equation}
   \cU_{-}(\alpha,\nu) := \{\bU \in \sa{k}{\C} : w_k - \bU\otimes \id_\cA  \succ 0\}.
\end{equation}
Then $\hat\bS_\up{\alpha,\nu}$ has an analytic continuation to $\cU_-$, and it maps $\cU_-$ diffeomorphically to $\cP_-$ given in Eq.~\eqref{eq:P_def} and~\eqref{eq:P_def_with_B}.
Further, $\bK_\up{\alpha,\nu}:\cP_-\to\cU_-$ is its analytic inverse in this domain. 
\end{theorem}

Note that since $\nu$ is not assumed to be supported on positive semi-definite matrices, a similar statement holds regarding the set
\begin{equation}
   \cU_+\equiv \cU_{+}(\alpha,\nu) := \{\bU \in \sa{k}{\C} : w_k - \bU\otimes \id_\cA  \prec 0\},
\end{equation}
where its image under $\hat\bS_\up{\alpha,\nu}$ is replaced by
\begin{equation}
\label{eq:def_P_+}
   \cP_+ := \{\bS\prec\bzero : \bS^{-1}+ \bT\prec\bzero \;\; \forall\,\bT\in\supp(\nu),\;\; \exists \bB \succ\bzero \;\textrm{s.t.}\;
  \cuD_\bS \bK(\bS)[\bB]\succ\bzero\; 
  \}.
\end{equation}

The set $\cU_-$ can be thought of as a generalization of the notion of `left of the support', in particular, for $k=1$, $\cU_- = \{u\in\R : u < \zeta_{-}\}.$  
This analogy will be more apparent in the proof of Theorem~\ref{thm:edge}, which we present next.

\paragraph{Deriving the variational formulas.}
\begin{proof}[Proof of Theorem~\ref{thm:edge}]
Fix $\alpha,\nu$ as in the theorem statement.
Clearly, for any $\bU\in\cU_-,$ we have
$(\bI_k\otimes \Pi)w_k \succ \bU \otimes  \Pi(\id_\cA).$
So for any $\xi = (\xi_1 ,\xi_2 , \dots ,\xi_k)$ with $\|\xi\|_{\cH^{k}} = 1,$
\begin{equation}
\langle\xi,
(\bI_k\otimes \Pi) w_k\;
 \xi\rangle_{\cH^{k}} \ge  \langle \xi, (\bU\otimes\id_\cH) \xi \rangle_{\cH^{ k}} = \sum_{i,j\in[k]} {\inner{\xi_i,\xi_j}}_{\cH} U_{ij} = \Tr( \bC \bU ),
\end{equation}
where we defined the Gram matrix $\bC := (\inner{\xi_i,\xi_j}_{\cH})_{i,j}$. Taking the spectral decomposition of the self-adjoint complex matrix $\bU = \sum_{i=1}^k d_iu_i u_i^*$, we then see that
\begin{equation}
   \Tr(\bC\bU) = \sum_{i=1}^k d_i u_i^* \bC u_i \ge \min_{i} d_i \cdot \Tr(\bC) = \lambda_{\min}(\bU),
\end{equation}
where we used that $\Tr(\bC) = \sum_i \|\xi_i\|^2 = 1.$ Combining the above two displays, taking the infimum over $\xi$  of unit norm, then the supremum over $\bU\in\cU_-$ yields
\begin{equation}
    \zeta_- \ge \sup_{\bU\in\cU_-}\lambda_{\min}(\bU).
\end{equation}
On the other hand, if $x < \zeta_{-}$, then $x\bI_k \in\cU_-.$ So choosing a sequence $x_m < \zeta_-$ so that $|x_m - \zeta_-| \to 0$ as $m\to\infty$ we have
\begin{equation}
   \zeta_- =  \lim_{m\to\infty} \lambda_{\min}(x_m \bI_k) \le \sup_{\bU\in\cU_-} \lambda_{\min}(\bU).
\end{equation}
Since by Theorem~\ref{thm:K_transform_real} we have
   $\cU_-  = \{
    \bU : \bU=\bK(\bS),\; \bS\in\cP_-\},$
    changing variables
    and combining the above two displays gives
   \begin{align}
       \zeta_{-} &= \sup_{\bU \in\cU_-}\lambda_{\min}(\bU)
       = \sup_{\bS\in\cP_-}  \lambda_{\min}\left(\bK(\bS)\right).
   \end{align}
\end{proof}

\begin{remark}[The right edge]
\label{remark:zeta_+}
As remarked previously, a similar formula for the right edge $\zeta_+$ can be directly obtained from Theorem~\ref{thm:edge} by replacing $\bT$ with $-\bT$.
This gives
\begin{equation}
\zeta_+ := \sup \supp(\mu_\up{\alpha,\nu})= \inf_{\bS \in \cP_{+}} \lambda_{\max}\left(\bK_\nu(\bS)\right),
\end{equation}
where $\cP_+$ was defined in~\eqref{eq:def_P_+}.
\end{remark}

We now present the proof of Theorem~\ref{thm:log}.
\begin{proof}[Proof of Theorem~\ref{thm:log}]
Fix $u<\zeta_{-}$ as in the theorem statement.
We suppress the dependence on $\alpha$ and $\nu$ for brevity and write, for example, $\hat\bS$ instead of $\hat\bS_{\up{\alpha,\nu}}$,
$g(\cdot)$ instead of $g_\up{\alpha,\nu}(\cdot)$,
and so on.
By Theorem~\ref{thm:K_transform_real}, we can re-write the variational formula of interest as
\begin{equation}
    \inf_{\bS\in\cP_-} g(\bS;u) = \inf_{\bU\in\cU} g(\hat\bS(\bU);u).
\end{equation}
Fix $\bU_0 \in\cU_-$ arbitrary. We will show that $g(\hat\bS(u\bI); u) \le g(\hat\bS(\bU_0);u)$.
Since $u <\zeta_-$, we have $u\bI\in \cU_-$, and therefore
 $g(\hat\bS(u\bI); u) =\inf_{\bU\in \cU_-} g(\hat\bS(\bU);u)$. This will then conclude the proof, since 
\begin{equation}
    g(\hat\bS(u\bI);u ) = k\int \log | \zeta - u| \mu(\de \zeta);
\end{equation}
this identity can be seen by differentiating both sides and taking $u\to -\infty$ to match the boundary 
conditions.

Define for $\eps\in(0,1],$ 
\begin{equation}
    \bU_\eps := \bU_0 - \eps\bDelta \quad\quad
    \bDelta := \bU_0 - u\bI_k
\end{equation}
and note that we have
\begin{align}
    w_k - \bU_\eps\otimes \id_\cA&= 
    w_k - ((1-\eps)\bU_0 + \eps u \bI_k) \otimes \id_\cA 
    =(1-\eps) \left( w_k - \bU_0\otimes \id_\cA\right) 
    + \eps \big(w_k - u\bI_k\otimes \id_\cA\big)
    \succ\bzero
\end{align}
for all $\eps\in(0,1].$
Now introduce the notation
    $m_\eps := w_k - (\bU_\eps\otimes\id_\cA)$
and compute the derivative.
\begin{align}
\frac{\partial}{\partial \eps}
g(\hat\bS(\bU_\eps);u)
 &\stackrel{(a)}{=}-\cuD_\bU g(\hat\bS(\bU_\eps);u)[\bDelta]\\
   &\stackrel{(b)}{=}
   -\frac1\alpha (\Tr\otimes \tau)
   \bigg(
  \Big( \big[-\alpha u\bI + \alpha \bK(\hat\bS(\bU_\eps))\big] \otimes  \id_\cA\Big)
  m_\eps^{-1}  (\bDelta \otimes \id_\cA) m_\eps^{-1}
   \bigg)\\
   &\stackrel{(c)}{=}
   -(1-\eps)(\Tr\otimes \tau)
   \big(
  \big( \bDelta \otimes  \id\big)
  m_\eps^{-1}  \big( \bDelta \otimes \id\big) m_\eps^{-1}
   \big)\\
   &\stackrel{(d)}{=}
   -(1-\eps)(\Tr\otimes \tau)
   \Big(
  \big(
  m_\eps^{-1/2}  \big( \bDelta\otimes \id\big) m_\eps^{-1/2} \big)^2
   \Big)\\
   &\le 0 
\end{align}
for all $\eps\in(0,1),$
where $(a)$ is by the chain rule, $(b)$ is by direct computation, $(c)$ is by Theorem~\ref{thm:K_transform_real}, whence
\begin{equation}
    \alpha\left(-u\bI + \bK(\hat\bS(\bU_\eps)) \right) = \alpha (\bU_\eps - u\bI) =\alpha (1-\eps) \bDelta,
\end{equation}
and in $(d)$ we used the fact that $\bU_\eps\in\cU$ to conclude that $m_\eps\succeq 0$.
(In fact, the inequality is strict whenever $\bDelta \neq 0$, i.e., $\bU_0\neq u\bI_k$).
Hence,
\begin{equation}
    g(\hat\bS(u \bI_k);u) = 
    g(\hat\bS(\bU_1);u)  \le
    g(\hat\bS(\bU_0);u)
\end{equation}
as desired.
\end{proof}
\begin{remark}[$\Phi$ to the right of the support]
\label{remark:log_+}
Once again, we can directly derive a formula for the logarithmic potential for $u > \zeta_+$  by symmetry,
\begin{equation}
  \int \log |\zeta - u| \mu_{\up{\alpha,\nu}}(\de \zeta) 
    = \frac1k \inf_{\bS\in\cP_{+}} g_\up{\alpha,\nu}(\bS;u).
\end{equation}
\end{remark}

\paragraph{Some useful consequences.}
As corollaries of the formulas above, we note a few direct consequences that are a priori not obvious. 
First we have the following characterization of measures $\mu_\up{\alpha,\nu}$ that are positively supported.
Such a question is for instance relevant in non-convex minimization problems such as minimizing the empirical risk $\hat R_n(\bTheta)$ mentioned in the introduction, since local minima correspond to points where the Hessian is positive semi-definite.
\begin{corollary}
Fix $\alpha\in(0,\infty)$, integer $k\ge1$, $\nu\in\cuP_0(\sa{k}{\C})$.
Then  $\supp(\mu_\up{\alpha,\nu})\subseteq(0,\infty)$ if and only
if there exist $\bS,\bB\succ\bzero$ satisfying all of the following
\begin{enumerate}
    \item 
    $
   \alpha^{-1} \bS^{-1}\bB\bS^{-1} - 
   \int \bT(\bI+\bS\bT)^{-1}\bB (\bI+\bT\bS)^{-1}\bT\, \nu(\de\bT)\succ\bzero$
   \item
$\int \bT(\bI+\bS\bT)^{-1} \nu(\de\bT) - \alpha^{-1} \bS^{-1}\succ\bzero.$
\item $\bS^{-1} +\bT \succ\bzero$ for all $\bT\in\supp(\nu).$
\end{enumerate}
\end{corollary}

Another useful consequence is  continuity of $\zeta_{-,\up{\alpha,\nu}}$ at $\nu$ whenever
we restrict to sequences $\nu_j$ whose supports are (approximately and eventually) contained in that of $\nu$.

\begin{corollary}
Fix $\alpha\in(0,\infty)$, integer $k\ge1$, $\nu\in\cuP_0(\sa{k}{\C})$. 
Let $\{\nu_j\}_{j\ge 1}$ be a sequence $\nu_j\in\cuP_0(\sa{k}{\C})$ such that
$\nu_j \stackrel{w}{\Rightarrow} \nu$ as $j\to\infty$, and,
for all $\delta>0$ there exists $j_0(\delta)$ such that
\begin{align}
j\ge j_0(\delta)\;\; \Rightarrow\;\;
\supp(\nu_j) \subseteq\bigcup_{\bT\in\supp(\nu)} \big\{\bX\in \sa{k}{\C}
:\;\; \bX\succeq \bT-\delta\bI\big\}\, .
\end{align}
Then
\begin{equation}
    \lim_{j\to \infty} \zeta_{-,\up{\nu_j,\alpha}} =  \zeta_{-,\up{\nu,\alpha}}.
\end{equation}
\end{corollary}
We note that in general, when $\nu_j\stackrel{w}{\Rightarrow} \nu$ without 
assumptions about the support $\supp(\nu_j)$, the conclusion of the above corollary 
can fail to hold. As an example, one can consider $\nu_j=a_j\delta_{-M}+\nu_j^+$,
where $\nu_j^+$ is supported on $(0,\infty)$ and $a_j\to 0$, but $M>0$ is sufficiently large.
\begin{proof}
    We will suppress the dependence on $\alpha$ in the notation.
    Fix $\eps >0$ and use Theorem~\ref{thm:edge} to find $\bS_\eps,\bB_\eps \succ\bzero$ and $\delta_\eps>0$
    satisfying 
    \begin{equation}
        \lambda_{\min}(\bK_\up{\nu}(\bS_\eps)) \ge \zeta_{-,\up{\nu}}  - \eps,
\qquad
\cuD_\bS \bK_\up{\nu}(\bS_\eps)[\bB_\eps] \succ \delta_\eps \bI_k,
\qquad
\bS_\eps^{-1} + \bT \succ \delta_\eps \bI_k \quad \forall \bT \in \supp(\nu),
    \end{equation}
Note that we used the assumption that $\nu$ is compactly supported to ensure that $\bS_\eps^{-1} + \bT \succ\bzero$ for all $\bT\in\supp(\nu)$
implies a uniform bound on $\bS_\eps^{-1} + \bT$.

The condition on the supports of $\nu_j$ implies that there exists $J_{\eps}<\infty$ such that 
$\bS_\eps^{-1} + \bT \succ \delta_\eps \bI_k/2$
for all $\bT \in \supp(\nu_j)$ provided $j\ge J_{\eps}$.  The convergence $\nu_j \Rightarrow \nu$ then implies that 
$\cuD_\bS \bK_\up{\nu_j}(\bS_\eps)[\bB_\eps]$ converges to $\cuD_\bS \bK_\up{\nu}(\bS_\eps)[\bB_\eps]$ in the operator norm.
Hence, by eventually increasing $J_{\eps}$, we have that, for all $j \ge J_\eps$,
\begin{equation}  
    \cuD_\bS \bK_\up{\nu_j}(\bS_\eps)[\bB_\eps] \succ \frac{\delta_\eps}{2} \bI_k,
\end{equation}
and therefore we conclude $\bS_\eps \in \cP_{-}(\nu_j)$ for all $j \ge J_\eps$.
Applying Theorem~\ref{thm:edge} to $\nu_j$ for $j \ge J_\eps$ now gives
\begin{equation}
   \zeta_{-,\up{\nu_j}} = \sup_{\bS \in \cP_{-}(\nu_j)} \lambda_{\min}(\bK_\up{\nu_j}(\bS))
   \ge  \lambda_{\min}(\bK_\up{\nu_j}(\bS_\eps)) 
\end{equation}
Taking the limit inferior, and using the fact that $\bK_\up{\nu_j}(\bS_\eps)$ converges in operator norm to
$\bK_\up{\nu}(\bS_\eps)$, we get 
\begin{equation}
   \lim\inf_{j\to\infty}\zeta_{-,\up{\nu_j}} \ge  \lambda_{\min}(\bK_\up{\nu}(\bS_\eps))
 \ge \zeta_{-,\up{\nu}}  - \eps\, .
 \end{equation}
 Letting  $\eps\to 0$, yields the lower bound 
  $\lim\inf_{j\to\infty}\zeta_{-,\up{\nu_j}} \ge \zeta_{-,\up{\nu}}$.
  
  To prove the matching upper bound, we note that the weak convergence of $\nu_j\Rightarrow\nu$
  implies $\lim_{j\to\infty}\hat\bS_{\up{\nu_j}}(z\bI)=\hat\bS_{\up{\nu}}(z\bI)$ for all $\Im(z)>0$,
  by a standard stability argument of the fixed point equation \eqref{eq:ST_FP}
  (see e.g. \cite{asgari2025local}). As a consequence
  $\mu_{(\nu_j)}\Rightarrow\mu_{(\nu)}$, which in turn implies  $\lim\sup_{j\to\infty}\zeta_{-,\up{\nu_j}} \le \zeta_{-,\up{\nu}}$.
\end{proof}

\section{Proof of Theorem~\ref{thm:K_transform_real}}

The remainder of the manuscript is dedicated to the proof of the 
Theorem~\ref{thm:K_transform_real}. It is perhaps useful to recall that 
$\bbH_+^k := \{\bB\in \C^{k\times k}: \Im(\bB) \succ\bzero\}$.

The starting point is to consider a more general form of the equation Eq.~\eqref{eq:ST_FP} for $\bZ\in\bbH_+^k.$ 
The next lemma establishes that the solution of this matrix equation is given by the Stieltjes transform 
$\hat\bS_\up{\alpha,\nu}$ of Eq.~\eqref{eq:free_ST_def}.
%
\begin{lemma}
\label{lemma:inv_on_upper_plane}
Fix integer $k \ge 1,$
let $\alpha \in (0,\infty)$, and $\nu\in\cuP_0(\sa{k}{\C})$.
 For any $\bZ \in\bbH_+^k$, 
 $\hat\bS_\up{\alpha,\nu}(\bZ)$ 
 of Eq.~\eqref{eq:free_ST_def}
 is the unique solution of
 \begin{equation}\label{eq:Matrix-Fixed-Point}
 \bK_{\up{\alpha,\nu}}(\bS) = \bZ,\quad\quad \bS\in\bbH_+^k.
 \end{equation}
\end{lemma}
%
%
%

Such result is fairly standard.
In order to prove that $\hat\bS_\up{\alpha,\nu}(\bZ)$ 
of Eq.~\eqref{eq:free_ST_def} is a solution of 
Eq.~\eqref{eq:Matrix-Fixed-Point}, a standard approach is to use a 
`leave-one-out' argument at finite $n,p$ to establish that the 
empirical Stieltjes transform $\hat\bS_n(\bZ)$ of Eq.~\eqref{eq:empirical_ST}
 satisfies the fixed-point equation in an approximate sense. 
 We refer to Appendix A of~\cite{asgari2025local} for such an argument in this setting.

Uniqueness is typically established by showing that the iterates of the 
fixed-point equation converge to a unique point from any initial condition.
See for example~\cite{helton2007operator} for a proof in the matrix-valued semi-circular setting. 
We will give a full proof of uniqueness in Appendix~\ref{app:proofs_FP} via a slightly different approach. 

As the statement of Theorem~\ref{thm:K_transform_real} suggests,
we will be interested in the Stieltjes transform evaluated at 
self-adjoint matrices $\bU\in \cU_-$, i.e., on the boundary of $\bbH_+^k.$
The next lemma carries out the extension to $\cU_-.$
Before stating it, let us first introduce the notation
\begin{equation}
    \cD_0 := \{\bS\succ\bzero : \bS^{-1} +  \bT \succ\bzero \quad \forall \bT \in\supp(\nu)\}
\end{equation}
to denote PSD matrices at which $\bK$ will be evaluated.

\begin{lemma}[Regularity of $\hat\bS_\up{\alpha,\nu}$ on $\cU_-$]
\label{lemma:regularity_free_ST}
Fix integer $k \ge 1,$
let $\alpha \in (0,\infty)$, and $\nu\in\cuP_0(\sa{k}{\C})$.
The Stieltjes transform $\hat\bS_\up{\alpha,\nu}(\bU)$ defined by 
Eq.~\eqref{eq:free_ST_def}
is analytic on an open set in $\sfM_{k}(\C)$ containing $\cU_-$.
In particular, for all $\bU\in\cU_-,$
\begin{equation}
\label{eq:lim_from_imaginary_matrices}
    \lim_{y\downarrow 0} \hat \bS_\up{\alpha,\nu}(\bU + i y\bI) = \hat \bS_\up{\alpha,\nu}(\bU).
\end{equation}
Consequently, if $\hat \bS_\up{\alpha,\nu}(\bU)\in\cD_0$ we have
\begin{equation}
\label{eq:K_S_U_is_U_on_D0}
    \bK_\up{\alpha,\nu}(\hat\bS_\up{\alpha,\nu}(\bU)) = \bU.
\end{equation}
Furthermore,
 the derivative of $\hat\bS_\up{\alpha,\nu}$ at $\bU\in\cU_-$ as an operator on 
 $\sa{k}{\C}$ is given by
\begin{align}
\label{eq:DhatS}    \De \hat\bS_\up{\alpha,\nu}(\bU)[\bDelta] &= \frac1\alpha(\bI_k \otimes \tau)\left(r_k(\bU)(\bDelta \otimes \id_\cA)r_k(\bU)\right) , \quad\bDelta\in\sa{k}{\C}\, ,\\
    r_k(\bU)&:= \big(w_k -(\bU \otimes \id_\cA)\big)^{-1}\, .\nonumber
\end{align}
Further,  $\De \hat\bS_\up{\alpha,\nu}(\bU)[\cdot]$ is non-singular and
\begin{equation}
    \langle \bDelta, \cuD\hat \bS_\up{\alpha,\nu}(\bU)[\bDelta]\rangle_\fro >0\quad\quad\forall
    \quad \bDelta\neq\bzero.
\end{equation}

\end{lemma}
\begin{proof}
We suppress the dependence on $\alpha,\nu$ in the notation.
Fix $\bU_0 \in \cU_{-}$, and define
$m_0:= w_k-\bU_0\otimes \id_\cA\in \sa{k}{\cA}.$
Set 
$c_0:=\|m_0^{-1}\|^{-1}>0.$
Then for any $\bZ\in \sfM_k(\C)$ with $\|\bZ-\bU_0\|<c_0/2$, we can write
$$
w_k-(\bZ\otimes \id_\cA)
=
m_0^{1/2}\Big(\bI - m_0^{-1/2}\big((\bZ-\bU_0)\otimes \id_\cA\big)m_0^{-1/2}\Big)m_0^{1/2}.
$$
Moreover,
$$
\big\|m_0^{-1/2}\big((\bZ-\bU_0)\otimes \id_\cA\big)m_0^{-1/2}\big\|
\le \|m_0^{-1}\|\,\|\bZ-\bU_0\|<\frac{1}{2}\, .
$$
Hence the middle factor is invertible since the associated Neumann series converges, and therefore
$$
\big(w_k-(\bZ\otimes \id_\cA)\big)^{-1}
=
m_0^{-1/2}\sum_{p=0}^\infty
\Big(m_0^{-1/2}\big((\bZ-\bU_0)\otimes \id_\cA\big)m_0^{-1/2}\Big)^p m_0^{-1/2}.
$$
The series converges absolutely in operator norm, uniformly on the
ball $B(\bU_0,c_0/2)\subset \sfM_k(\C)$. It follows that
$\bZ\mapsto \big(w_k-(\bZ\otimes \id_\cA)\big)^{-1}$
is analytic on this ball. Since $(I_k\otimes \tau)$ is continuous and 
linear, we conclude that $\hat\bS(\bZ)$
extends analytically to $B(\bU_0,c_0/2)$. As $\bU_0\in \cU_-$ 
was arbitrary, $\hat \bS$ admits an analytic continuation to $\cU_-$.
Equation~\eqref{eq:lim_from_imaginary_matrices} follows because any analytic function is also continuous.
Finally, Eq.~\eqref{eq:K_S_U_is_U_on_D0} follows because 
 $\bK(\hat\bS(\bU+i y \bI)) = \bU + i y \bI$ for all $y>0$ by
 Lemma~\ref{lemma:inv_on_upper_plane}, $\hat\bS(\bU+i y \bI)\to \hat\bS(\bU)$ as $y\downarrow 0$
 by the previous point, and $\bS\mapsto \bK(\bS)$ is continuous at $\hat\bS(\bU)\in\cD_0$
 because $\inf_{\bT\in\supp(\nu)}\lambda_{\min}(\bS^{-1}+ \bT) > 0$, $\lambda_{\min}(\bS)>0$ 
 for any $\bS\in\cD_0$.

Finally, Eq.~\eqref{eq:DhatS} for the derivative follows by simple calculus.
To show that it is non-singular, use positivity of $r_k(\bU)$ to write, 
for any $\bDelta\in \sa{k}{\C}$,
\begin{align}
\langle \bDelta, \cuD\hat \bS(\bU)[\bDelta]\rangle_\fro
&=
\frac1\alpha (\Tr\otimes \tau)\Big((\bDelta\otimes \id_\cA)\,r_k(\bU)\,(\bDelta\otimes \id_\cA)\,r_k(\bU)\Big)\\
&=
\frac1\alpha (\Tr\otimes \tau)\Big(
\big(r_k(\bU)^{1/2}(\bDelta\otimes \id_\cA)r_k(\bU)^{1/2}\big)^2
\Big)\ge 0.
\end{align}
If this quantity is zero, then by faithfulness of $(\Tr\otimes \tau)$ and positivity of $r_k$, 
we must have $\bDelta = \bzero$.
\end{proof}

\paragraph{The image of $\cU_-$ is contained in $\cD_0$.}
We will now proceed to show that the image of $\cU_-$ under $\hat\bS_\up{\alpha,\nu}$ is contained in the domain of $\bK,$ specifically, in $\cD_0.$
Observe that for $\bU  = u\bI$ with $u\to-\infty,$ we have $\bU \in\cU_-$ eventually when $\nu$
 is compactly supported. Then since $\hat \bS_\up{\alpha,\nu}(\bU) \to \bzero,$ we will have 
 $\hat\bS_\up{\alpha,\nu}(\bU) \in \cD_0$ eventually. The next lemma shows that this property extends to all of $\cU_-$, at least when $\nu$ is a discrete measure.

\begin{lemma}[$\hat\bS_\up{\alpha,\nu}(\cU_-) \subseteq \cD_0$: discrete $\nu$]
\label{lemma:image_Sstar_D_discrete}
Fix $\alpha \in(0,\infty)$.
Assume $\nu = \sum_{i=1}^m \delta_{\bT_i} p_i$ for $\sum_{i=1}^m p_i = 1$ with
$p_i > 0$.
Then for any $\bU \in\cU_-$, we have 
\begin{equation}
    \hat\bS_\up{\alpha,\nu}(\bU) \in \cD_0.
\end{equation}
\end{lemma}
\begin{proof}
Once again, we suppress the dependence on $\alpha,\nu$ in the notation.
Since $\nu$ is compactly supported, it follows from the definition
\eqref{eq:wk-def} that $\|w_k\|\le C$, for some constant $C>0$. 
Hence, there exists $u_0 <0$ sufficiently negative such that
$u_0\bI_k \in\cU_-$ and $\hat\bS(u_0 \bI) \in \cD_0$.
Let $\cC_{0}:=\{\bU \prec u_0\bI\}$. 
Since the inverse is monotone with respect to Loewner ordering for positive operators, we have
$\hat\bS(\bU) \preceq\hat\bS(u_0\bI)$ for all $\bU\in\cC_0,$ so that $\hat\bS(\bU) \in \cD_0$ as well.
 Now fix some $\bU_0 \in \cC_0$ and arbitrary $\bU_1\in\cU_-$ and suppose for contradiction that $\bS_1:= \hat\bS(\bU_1) \not \in \cD_0.$ That is, suppose there exists some $\bT_1\in\supp(\nu)$ such that 
\begin{equation}
    \lambda_{\min}\left( \bS_1^{-1}+ \bT_1 \right) \le 0.
\end{equation}
For $\delta\in[0,1]$, let $\bU_\delta := \delta \bU_1 + (1-\delta)\bU_0.$
Clearly, $\bU_\delta \in \cU_-$  for all $\delta \in [0,1],$ and
\begin{equation}
    \lim_{\delta \to 0} \hat\bS(\bU_\delta) \in\cD_0,\quad\quad
    \lim_{\delta \to 1} \hat\bS(\bU_\delta) = \hat\bS(\bU_1) \not \in \cD_0
\end{equation}
since $\delta \mapsto \hat\bS(\bU_\delta)$ is continuous, so there must exist some $\delta_0\in(0,1)$ such that
\begin{equation}
    \hat\bS(\bU_\delta) \in \cD_0 \quad\textrm{for}\;\; \delta < \delta_0\quad\quad\textrm{and}\quad\quad \lambda_{\min}( \hat\bS(\bU_{\delta_0})^{-1} +  \bT_i
    ) = 0 \quad\textrm{for some }i\in[m],
\end{equation}
and since $\nu(\{\bT_i\}) =: p_i > 0$ for all $i\in[m]$ by assumption, we have, using the 
definition of $\bK$ in Eq.~\eqref{eq:K_transform},
\begin{equation}
    \lim_{\delta\uparrow \delta_0} \|\bK(\hat\bS (\bU_\delta))\|_\op = \infty.
\end{equation}
But, this contradicts 
 Lemma~\ref{lemma:regularity_free_ST}
 which guarantees that for $\delta<\delta_0$, we have
\begin{equation}
    \bK(\hat\bS(\bU_\delta)) =  \bU_\delta.
\end{equation}
\end{proof}

%
%

To extend the applicability of the previous argument to general compactly supported $\nu$, 
we state the following fact that allows us to approximate a measure $\nu$ by a discrete one.
\begin{lemma}[Discrete approximation lemma]
\label{lemma:discrete_approx}
Fix $\alpha \in(0,\infty)$.
   For any compactly supported probability measure $\nu\in\cuP(\sa{k}{\C})$, 
   and all $\eps > 0$,  
   there exists a measure $\nu^\up{\eps}$ and a set $\cC^\up{\eps} \subseteq \C^{k\times k}$ 
   such that the following hold:
   \begin{enumerate}
       \item
       \label{item:1_approx}
       $\cC^\up{\eps} \subseteq \supp(\nu)$, has finite cardinality, and 
       $$
       \label{eq:closed_cover}
       \sup_{\bY \in \supp(\nu)} \inf_{\bT \in \cC^\up{\eps}} \|\bT - \bY\|_\op \le \eps\, .$$
       \item
       \label{item:2_approx}
       The measure $\nu^\up{\eps}$ is supported on $\cC^\up{\eps}$, with weights
       $p_\eps(\bT) =\nu^\up{\eps}(\{\bT\})$ satisfying:
       \begin{equation}
       \nu^\up{\eps} :=\sum_{\bT \in \cC^\up{\eps}} p_\eps(\bT)\delta_{\bT},\quad\quad \sum_{\bT\in \cC^\up{\eps}} p_\eps(\bT) = 1 ,\quad\quad 
       p_\eps(\bT) > 0 \quad\forall \bT \in \cC^\up{\eps}.
       \end{equation}
       \item
       \label{item:3_approx}
       Recall the definition of $w_k$ from Eq.~\eqref{eq:wk-def},
       and let, for $\eps >0$,
       $w_k^\up{\eps}$ denote
       the free element 
       defined by Eq.~\eqref{eq:wk-def} with $\nu$ replaced by $\nu^\up{\eps}$,

       Then for any $\bU \in\cU_-(\alpha,\nu)$,
        there exists $\eps_0(\bU) >0$ such that for all $\eps < \eps_0(\bU)$,
       \begin{equation}
        w_k^\up{\eps}  - (\bU \otimes \id_\cA) \succ 0.
       \end{equation}

       \item \label{item:4_approx}
       For any $\bU \in\cU_- (\alpha,\nu) \cap \cU_-(\alpha,\nu^\up{\eps})$
        \begin{equation}
            \|\hat \bS_\up{\alpha,\nu}(\bU)  - \hat \bS_\up{\alpha,\nu^\up{\eps}}(\bU)\|_\op \to 0
        \end{equation}       
        as $\eps\to 0$.

   \end{enumerate}
   
\end{lemma}
\begin{proof}
    For $\eps>0$,
    throughout, let $B_{\eps}(\bT_i)$ be the open norm ball of radius $\eps$ around $\bT_i$. Let $\cC_0 := \supp(\nu).$
   Construct a \emph{closed} $\eps-$cover $\cC^\up{\eps}$
  (i.e., satisfying Eq.~\eqref{eq:closed_cover})
   of the compact set $\cC_0$ with balls $\overline B_\eps(\bT_i)$ 
   satisfying  
   \begin{equation}
       \|\bT -\bY\|_\op > \eps\quad\quad \forall \bT,\bY\in \cC^\up{\eps},\; \bT\neq\bY.
   \end{equation}
   Concretely, start with $\bT_1\in \cC_0$ arbitrary, then for each $i> 1$, arbitrarily choose $\bT_i \in \cC_0 - \bigcup_{j< i}\overline B_\eps(\bT_j)$ until $\cC_0$ is covered by $\{\overline B_\eps(\bT_i)\}$. 
   This procedure halts after a finite number $m$ of steps  by compactness (For example, see Lemma~4.2.6 of
   ~\cite{vershynin2025high} and the remark that follows it).
   Choose $\cC^\up{\eps} = 
   \{\bT_i:i\le m\}$. 
   Item~\textit{1} holds by construction

   Define
   \begin{equation}
       S_1 := \overline B_\eps(\bT_1),\quad\quad S_i := \overline B_\eps(\bT_i) \setminus \bigcup_{j < i} \overline B_{\eps}(\bT_j),\quad i > 1
   \end{equation}
   and set 
   \begin{equation}
   \label{eq:p_defs}
   p_\eps(\bT_i) := \nu(S_i).
   \end{equation}
Since $\{S_i\}_i$ are disjoint, it is easy to see that
\begin{equation}
    \sum_{\bT\in \cC^\up{\eps}} p_\eps(\bT) =  \nu(\supp(\nu)) = 1.
\end{equation}
By construction, for each $i\in[m]$, there exists $\delta > 0$ so that 
$$B_\delta(\bT_i)\cap \bigcup_{j < i} \overline B_\eps(\bT_j)  = \emptyset.$$
And since $\bT_i \in \cC_0$, we must have 
\begin{equation}
    0 < \nu( B_\delta(\bT_i)) \le p_\eps(\bT_i),
\end{equation}
showing that Item~\textit{2} holds.

For the remaining items, for $\bA\in\C^{k\times k}$, fix $\eps>0$ arbitrarily and let
\begin{eqnarray}
    \bY_\eps(\bA) :=  \sum_{j\in[m]} \bT_j \one_{\bA \in S_j},
\end{eqnarray}
and define i.i.d. random variables 
$(\hat\bT_i, \hat\bY_i)_{i\in[n]}$ as follows:
for each $i$,
\begin{equation}
    \hat\bT_t \sim \nu, \qquad \text{and} \qquad \hat\bY_i := \bY_\eps( \hat\bT_i).
\end{equation}
Define the following empirical measures:
\begin{equation}
    \hat\nu_n := \frac1n\sum_{i=1}^n \delta_{ \bT^\up{i}},\qquad
    \hat\nu_n^\up{\eps} := \frac1n\sum_{i=1}^n \delta_{ \bY^\up{i}}
\end{equation}
This guarantees that,
almost-surely,
$\supp(\hnu_n)\subseteq \supp(\nu), \supp(\hnu_n^\up{\eps})\subseteq \cC^\up{\eps}$ for all $n$, 
and as $n\to\infty$,  we have 
$\hnu_n\stackrel{w}{\Rightarrow} \nu$,
\begin{equation}
    \hnu_n^\up{\eps}(S_j)  = \frac1n \#\{i : \hat \bY_i \in S_j\} =
    \frac1n \#\{i : \hat  \bT_i \in S_j\}
    \to \nu(S_j) = p_\eps(\bT_j),\qquad \textrm{for all}\; j\in[m],
\end{equation}
and further, for any bounded continuous function $f:\C^{k\times k}\times \C^{k\times k}\to\R$,
\begin{equation}
    \frac1n\sum_{i=1}^n f(\bT^\up{i}, \bY^\up{i}) \to \E[f(\bT, \bY_\eps(\bT))],\qquad
    \textrm{where} \qquad \bT \sim \nu.
 \end{equation}

 Now let $(\cA,\tau)$ be a $C^*$-probability space with Marchenko-Pastur elements $m$, as in Section~\ref{sec:proofs_main}, 
free from the algebra generated by $(t_{i,j},t^\up{\eps}_{i,j})_{i,j\in[k]}$, which are 
in turn defined by
\begin{equation}
\tau \bigg( P\big( \{t_{i,j}\}_{i,j\in[k]}, \{t^\up{\eps}_{i,j}\}_{i,j\in[k]} \big) \bigg)
= \int 
P\big( \{\bT_{i,j}\}_{i,j\in[k]}, \{(\bY^\up{\eps}(\bT))_{i,j}\}_{i,j\in[k]} \big)
 d\nu(\bT)
\end{equation}
for any non-commutative polynomial $P$,
so that $w_k = (m^{1/2} t_{i,j} m^{1/2})_{i,j\in[k]}$ 
and $w_k^\up{\eps} = (m^{1/2} t_{i,j}^\up{\eps} m^{1/2})_{i,j\in[k]}$ in the sense of tracial moments.

Now let $\bX\in\R^{n\times p}$ be a matrix with i.i.d. standard Gaussian entries so that it satisfies the 
 assumptions in Proposition~\ref{prop:asymptotic_FP}, and 
 define
\begin{align}
    \bW_k^{\up{\eps}} := \frac1n(\bX \otimes \bI_k)^* \Diag(\hat\bY_i)_{i\le n} (\bX \otimes \bI_k).
\end{align}
Since  $\bW_k= n^{-1}(\bX \otimes \bI_k)^* \Diag(\hat\bT_i)_{i\le n} (\bX \otimes \bI_k)$,
\begin{align}
\label{eq:bound_on_W_k}
 \|\bW_k^{\up{\eps}} - \bW_k\|_\op
 \le  \frac{\|\bX\|^2_\op}{n}  \max_{i\in[n]} \|\hat\bT_i - \hat\bY_i\|_\op
\le  \frac{\|\bX\|^2_\op}{n} \max_{j\in[m]} \max_{\tilde \bT \in S_j} \|\bT_j-\tilde \bT\|_\op \\
\end{align} 
and we can bound the difference
\begin{align}
\|w_k - w_k^\up{\eps}\|_{\cB(\cH^k)} 
 &\stackrel{(a)}{=} \lim_{q\to\infty}
 \Big[
\Big(\frac1k \Tr \otimes \tau\Big)(w_k - w_k^\up{\eps})^{2q}
\Big]^{1/(2q)}\\
&\stackrel{(b)}{=}
\lim_{q\to\infty}
\alpha^{1/(2q)}
 \lim_{n\to\infty}
\E\Big[\Big(\frac1k \Tr \otimes \frac1n\Tr\Big)\big(\bW_k - \bW_k^\up{\eps}\big)^{2q}\Big]^{1/(2q)}\\
&\stackrel{(c)}{\le}
  \lim_{q\to\infty}
  \lim_{n\to\infty} 
  \Big(\alpha\frac{p}{n}\Big)^{1/(2q)}
\E\big[\|\bW_k - \bW_k^\up{\eps}\|_\op^{2q}\big]^{1/(2q)}\\
&\stackrel{(d)}{\le}
    \lim_{q\to\infty} \lim_{n\to\infty} 
  \Big(\alpha\frac{p}{n}\Big)^{1/(2q)}
\frac1n (\sqrt{n} + \sqrt{p} + 2\sqrt{q})^2 \eps
\\
&\le C(\alpha)\, \eps,
\label{eq:bound_on_w_k_w_k_eps}
\end{align}
where in $(a)$ we used Proposition~\ref{prop:spectra_match}, in $(b)$ we used 
Eq.~\eqref{eq:free_determinism}, in $(c)$ we used monotonicity of $L^p$ norms, and 
in $(d)$ we used a standard bound on the operator norm of Gaussian matrices and Eq.~\eqref{eq:bound_on_W_k}.
Since $\eps>0$ was arbitrary, this shows that Item~\textit{3} holds.

Finally, Item~\textit{4} holds by applying the resolvent identity 
\begin{equation}
   \hat \bS_\up{\alpha,\nu}(\bU) - \hat \bS_\up{\alpha,\nu^\up{\eps}}(\bU) 
   = 
   \frac1\alpha (\bI_k\otimes \tau)\Big((w_k^\up{\eps} - \bU \otimes \id_\cA)^{-1} (w_k- w_k^\up{\eps})(w_k - \bU \otimes \id_\cA)^{-1} \Big),
\end{equation}
using Item~\ref{item:3_approx}
to uniformly bound $\|(w_k^\up{\eps} - \bU \otimes \id_\cA)^{-1}\|, \|(w_k - \bU \otimes \id_\cA)^{-1}\|$
uniformly for $\eps$ sufficiently small along with 
Eq.~\eqref{eq:bound_on_w_k_w_k_eps}.

\end{proof}

\begin{lemma}[$\hat\bS_\up{\alpha,\nu}(\cU_-)\subseteq\cD_0$ for general measures]
\label{lemma:image_Sstar_D_general}
Let $\alpha \in (0,\infty)$.
Assume $\nu \in\cuP_0(\sa{k}{\C})$ (i.e., $\nu$ is a compactly supported measure).
Then for any $\bU \in\cU_-$, we have 
\begin{equation}
    \hat\bS_\up{\alpha,\nu}(\bU) \in \cD_0.
\end{equation}
\end{lemma}
\begin{proof}
Fix $\bU \in\cU_-$ so that $w_k - (\bU \otimes  \id_\cA) \succ \bzero$.
For any $\eps >0$, denote $w_k^\up{\eps}$ 
the free element defined by Eq.~\eqref{eq:wk-def} with $\nu^\up{\eps}$ 
instead of $\nu$.
By
Item~\textit{\ref{item:3_approx}} of Lemma~\ref{lemma:discrete_approx},
we can choose some $\eps := \eps(\bU) >0$, depending on $\bU$, so that 
$w_k^\up{\eps} - (\bU \otimes  \id_\cA) \succ 0$.
Denote  $\cU_-^\up{\eps} := \{\bU\in \sa{k}{\C}: 
w_k^\up{\eps} - (\bU \otimes  \id_\cA) \succ 0\}$ and 
$\hat\bS^\up{\eps} (\,\cdot\,) := \hat\bS_{\up{\nu^\up{\eps},\alpha}}(\,\cdot\,)$.

Item~\ref{item:2_approx} of Lemma~\ref{lemma:discrete_approx} guarantees 
that the conditions of Lemma~\ref{lemma:image_Sstar_D_discrete} hold, 
from which we conclude that
\begin{equation}
\label{eq:smallest_eval_discrete}
   \lambda_{\min} \big(\hat\bS^\up{\eps}(\bU)^{-1}+ \tilde\bT\big) > 0 \quad\quad\textrm{for all}\quad\quad \tilde\bT\in\supp(\nu^\up{\eps}).
\end{equation}
Therefore, for any $\bT\in\supp(\nu),$ we have by Item~\textit{\ref{item:1_approx}} of Lemma~\ref{lemma:discrete_approx},
\begin{equation}
\label{eq:smallest_eval_discrete_2}
    \lambda_{\min}\big( \hat\bS^\up{\eps}(\bU)^{-1} +\bT\big) \ge  -\eps.
\end{equation}
Now Item~\textit{\ref{item:4_approx}} of
Lemma~\ref{lemma:discrete_approx} along with 
the fact that $\bU \in \cU_-\cap\cU_-^\up{\eps}$ guarantees that $\hat\bS^\up{\eps}(\bU) \to \hat\bS(\bU)$ as $\eps\to 0$.
Combining with the bound
of Eq.~\eqref{eq:smallest_eval_discrete_2}, we obtain
\begin{equation}
    \lambda_{\min}\big( \hat\bS(\bU)^{-1} +\bT\big) \ge 0.
\end{equation}

Since $\bU\in\cU_-$ was arbitrary,  we have shown that for all $\bU\in\cU_-$, $\hat\bS(\bU) \in \overline \cD_0.$
Then since Lemma~\ref{lemma:regularity_free_ST} guarantees that $\hat\bS$ has a non-singular
 derivative at $\bU \in\cU_-$,
the inverse function theorem in turn guarantees that 
$\hat\bS(\cU_-)$ is an open set, and is hence contained in $\cD_0$.
\end{proof}

\paragraph{An alternative characterization of $\cP_-$.}
Before continuing on to show that $\hat\bS_\up{\alpha,\nu}(\cU_-) \subseteq \cP_-$, we will give a characterization of $\cP_-$ that will be useful in the proof.
We associate to each $\bS \in\cD_0$ the following operator mapping matrices to matrices:
\begin{equation}
    \cuT_\bS(\bDelta) := \alpha\int \bZ(\bI + \bZ)^{-1} \bDelta \bZ(\bI + \bZ) ^{-1}\nu(\de\bT)\quad\quad \bZ := \bS^{1/2}\bT\bS^{1/2}.
\end{equation}
Since $\bZ$ commutes with $(\bI +\bZ)^{-1}$, $\cuT_\bS$ is self-adjoint as a bounded linear operator on the space of self-adjoint matrices 
endowed with the trace inner product. Further, it is immediate to see
that $\<\bDelta, \cuT_\bS(\bDelta)\> \ge 0$ for all $\bDelta\in\sa{k}{\C}$, i.e. $\cuT_\bS$ is 
positive semidefinite.

Our interest in this operator comes from its relationship with the derivative of the function $\bK\equiv \bK_\up{\alpha,\nu}$, namely, observe that for all $\bS \in\cD_0$,  we have
\begin{equation}
\label{eq:DK_to_cuT}
\bS^{1/2}\cuD_\bS \bK(\bS)[\bDelta]\bS^{1/2} =  \frac1\alpha \left( \bS^{-1/2} \bDelta \bS^{-1/2}   - \cuT_\bS(\bS^{-1/2}\bDelta \bS^{-1/2})\right).
\end{equation}
Hence,  for any $\bDelta\in\sa{k}{\C}$, 
defining $\bA :=  \bS^{-1/2}\bDelta\bS^{-1/2}$,
we obtain
\begin{equation}
\label{eq:DK_to_cuT_norm}
   \langle
    \bDelta,\cuD_\bS \bK(\bS)[\bDelta] \rangle_\fro
   = 
   \frac1\alpha \big(
   \langle \bA,\bA\rangle - \langle \bA ,\cuT_\bS(\bA) \rangle
   \big).
\end{equation}
Therefore, using the fact that $\cuT_\bS$ is positive semidefinite, 
$\cP_-$ can be alternatively written as
\begin{equation}
    \cP_- = \{\bS\succ\bzero : \|\cuT_\bS\|_{F\to F} <1 \} \cap \cD_0\, .
\end{equation}

The following lemma gives a further characterization that turns out to 
be useful.
\begin{lemma}
\label{lemma:set_characterization}
For $\alpha\in(0,\infty)$, integer $k \ge 1,$ 
$\nu\in\cuP_0(\sa{k}{\C}),$ 
define
\begin{equation}
\tilde\cP_- := \{\bS \succ\bzero : \exists \bB\succ\bzero\;\;\textrm{s.t.}\;\; \cuT_\bS(\bB)\prec  \bB
\} \cap \cD_0.
\end{equation}
Then $\cP_- = \tilde\cP_-$.
\end{lemma}
\begin{proof}

The inclusion $\cP_- \subseteq \tilde\cP_-$ holds more generally for linear self-adjoint operators. To see this,
fix $\bS\in\cP_-$ and any $\bC\succ\bzero$ arbitrary, then define 
\begin{equation}
    \bB := \left(\bI + \sum_{p=1}^\infty \cuT_\bS^p\right)(\bC)
\end{equation}
where $\cuT_\bS^p$ denotes the $p-$fold composition of $\cuT_\bS$.
Since $\|\cuT_\bS\|_{\fro\to\fro} < 1$,
the operator in the above display is bounded in $\|\cdot\|_{\fro\to\fro}$ since the sum is geometric.
Clearly, $\cuT_\bS$ maps PSD matrices to PSD matrices so $\bB\succ\bzero.$
Meanwhile, by linearity
\begin{equation}
    \cuT_\bS(\bB) = \sum_{p=1}^\infty \cuT_\bS^p(\bC) = \bB - \bC \prec \bB
\end{equation}
since $\bC \succ\bzero.$ i.e., we obtain the estimate
    $\bzero \preceq \cuT_\bS(\bB) \prec \bB$,
    which implies $\bS\in\tilde\cP_-.$

To show the opposite inclusion, fix $\bS\in\tilde\cP_-$,
and introduce the expectation notation
 \begin{equation}
     \cuT_\bS(\bDelta) = \alpha \E[\bZ (\bI+\bZ)^{-1} \bDelta \bZ (\bI+\bZ)^{-1}] ,\quad\quad \bZ := \bS^{1/2} \bT\bS^{1/2}\, .
 \end{equation}
 Here the expectation is with respect to the pushforward of the probability measure $\nu$.
We will first obtain a bound on the norm $\|\cuT_\bS^p\|_{\fro\to\fro}$ in order to bound the spectral radius of $\cuT_\bS$.

To do so, fix $\bDelta \in\sfM_{k}(\C)_{\textrm{sa}}$ and let
\begin{equation}
    \bT_1,\dots,\bT_p,\tilde\bT_1,\dots,\tilde\bT_p
\end{equation}
i.i.d. copies of $\bT$ from $\nu$, and denote
\begin{equation}
   \bA_i := \bZ_i (\bI+\bZ_i)^{-1} \quad \bZ_i := \bS^{1/2}\bT_i\bS^{1/2}.
\end{equation}
Similarly define $\tilde \bA_i,\tilde \bZ_i$ using $\tilde\bT_i.$ Then from the definition of $\cuT_\bS$, we see that
\begin{align}
\cuT_\bS^p(\bDelta)  &= \alpha^p\E\left[\bA_p\cdots \bA_1 \bDelta \bA_1 \cdots\bA_p  \right]
\end{align}
Using  the fact that $2\Tr(\bDelta\bQ\bDelta\bQ^{*})\le
\Tr(\bDelta^2\bQ\bQ^{*})+\Tr(\bDelta^2\bQ\bQ^{*})$, we can bound
\begin{align}
   \|\cuT_\bS^p(\bDelta)\|_\fro^2  &=
   \alpha^{2p}\E\left[
   \Tr\left(\bA_p\cdots \bA_1 \bDelta \bA_1 \cdots\bA_p  
   \tilde\bA_p\cdots \tilde\bA_1 \bDelta \tilde\bA_1 \cdots\tilde\bA_p  
   \right)
   \right]\\
   &\le
   \alpha^{2p}
   \E\left[
   \Tr\left(\bDelta \bA_1 \cdots\bA_p  
   \tilde\bA_p\cdots \tilde\bA_1  \tilde\bA_1 \cdots\tilde\bA_p  
   \bA_p\cdots \bA_1 \bDelta
   \right)
   \right].
\end{align}
Since the matrix inside the trace is PSD, for any $\bB\succ\bzero,$ we have
\begin{align}
   &\E\left[
   \Tr\left(\bDelta \bA_1 \cdots\bA_p  
   \tilde\bA_p\cdots \tilde\bA_1\bB^{1/2}\bB^{-1}\bB^{1/2}  \tilde\bA_1 \cdots\tilde\bA_p  
   \bA_p\cdots \bA_1 \bDelta
   \right)
   \right]\\
   &\le 
   \|\bB^{-1}\|_{\op}
   \E\left[
   \Tr\left(\bDelta \bA_1 \cdots\bA_p  
   \tilde\bA_p\cdots \bB^{1/2} \E[\bB^{-1/2} \tilde\bA_1\bB\tilde\bA_1\bB^{-1/2}]\bB^{1/2} \cdots\tilde\bA_p  
   \bA_p\cdots \bA_1 \bDelta
   \right)
   \right]\\
   &\le\|\bB^{-1}\|_\op
   \|\E[\bB^{-1/2} \tilde\bA_1\bB\tilde\bA_1\bB^{-1/2}]\|_\op
   \E\left[
   \Tr\left(\bDelta \bA_1 \cdots\bA_p  
   \tilde\bA_p\cdots \tilde\bA_2\bB\tilde\bA_2 \cdots\tilde\bA_p  
   \bA_p\cdots \bA_1 \bDelta
   \right)
   \right].
\end{align}
Repeating this argument $2p-1$ many times, we obtain that, for any $\bB\succ\bzero$,
\begin{equation}
    \|\cuT_\bS^p(\bDelta)\|_\fro^2 \le\alpha^{2p} \|\bB^{-1}\|_\op \|\bB\|_\op 
    \|\E[\bB^{-1/2}\bA\bB\bA\bB^{-1/2}]\|_\op^{2p} \|\bDelta\|_\fro^2,
\end{equation}
i.e.,
\begin{align}
\label{eq:F_to_F_bound_composition_A}
    \|\cuT_\bS^p\|_{\fro\to\fro} \le  \alpha^{p} \|\bB\|_\op^{1/2}\|\bB^{-1}\|_\op^{1/2} 
    \|\E[\bB^{-1/2}\bA\bB\bA\bB^{-1/2}]\|_\op^{p},
\end{align}  
where $\bA = \bZ(\bI + \bZ)^{-1}$.
Consequently, if there is a $\bB\succ\bzero$ satisfying $\cuT_\bS(\bB)\prec\bB$, then for some $\delta\in(0,1),$
\begin{equation}
\|\E[\bB^{-1/2} \bA \bB\bA \bB^{-1/2}]\|_\op =
    \frac1{\alpha}\|\bB^{-1/2} \cuT_\bS(\bB) \bB^{-1/2}\|_\op < \frac{\delta}{\alpha}.
\end{equation}
Then since $\cuT_\bS^p$ is self-adjoint,  combining the last bound with 
 with Eq.~\eqref{eq:F_to_F_bound_composition_A}, we obtain that
\begin{equation}
   \|\cuT_\bS\|_{\fro\to\fro} =\lim_{p\to\infty}  \|\cuT_\bS^p\|_{\fro\to\fro}^{1/p} 
   \le \delta \,\lim_{p\to\infty} C^{1/p} < 1\, .
\end{equation}
\end{proof}

\paragraph{The image of $\cU_-$ is contained in $\cP_-$.}
Finally, we return to show that  $\hat\bS_\up{\alpha,\nu}(\cU_-)\subseteq\cP_-.$
\begin{lemma}
\label{lemma:image_U_in_P}
Fix  $k \ge 1$ integer,
$\alpha \in (0,\infty)$, and $\nu\in\cuP_0(\sa{k}{\C})$.
Then for any $\bU \in\cU_-$, we have 
\begin{equation}
    \hat\bS_\up{\nu;\alpha}(\bU) \in \cP_-.
\end{equation}
\end{lemma}
\begin{proof}
Fix $\bU \in\cU_-$ throughout the proof, and consider for $y \in(0,1)$,
\begin{equation}
    \bZ_y := \bU + i y \bI,
\end{equation}
and let 
\begin{equation}
    \bS_y := \hat\bS(\bZ_y)= \bA_y + i \bB_y \in \bbH_+^k.
\end{equation}
Recall that 
\begin{equation}
    \bK(\bS_y) =  \bZ_y\, .
\end{equation}
Using Lemma~\ref{lemma:inv_on_upper_plane}
and taking the imaginary part of both sides to get
\begin{equation}
  y \bI = \frac1\alpha \bS_y^{-1} \bB_y \bS_y^{*-1} - \E[(\bI + \bT \bS_y)^{-1}\bT \bB_y \bT (\bI + \bS_y^* \bT)^{-1}],
\end{equation}
which, after re-arranging, gives
\begin{equation}
\bB_y^{-1/2}\bS_y \E[(I + \bT \bS_y)^{-1}\bT \bB_y \bT (\bI + \bS_y^* \bT)^{-1}] \bS_y^*\bB_y^{-1/2}
  = \frac1\alpha\bI - y\bB_y^{-1/2} \bS_y \bS_y^*\bB_y^{-1/2}.\label{eq:Rearranged}
\end{equation}

Define $r_y = \big(w_k -(\bZ_y \otimes \id_\cA)\big)^{-1}\in \mat{k}{\C}\otimes\cA$ and note that
\begin{align}
   \alpha \bB_y &= \Im\circ (\bI_k \otimes \tau)\left(r_y\right)
  =  y\cdot(\bI_k \otimes \tau)(r_y r_y^*).
\end{align}
Noting that
\begin{equation}
    \lim_{y\to 0} r_y  = r_0(\bU):= 
    \big(w_k -(\bU \otimes \id_\cA)\big)^{-1}
   \succ \bzero,\quad\quad
   \lim_{y\to 0} \bS_y =  \hat\bS(\bU)\succ\bzero,
\end{equation}
we have $\lim_{y\to 0}\bB_y/y=\bB_0 := \alpha^{-1}(\bI_k \otimes \tau)(r_0(\bU)^2)$ and therefore
\begin{equation}
    \lim_{y\to 0}y\bB_y^{-1/2} \bS_y \bS_y^*\bB_y^{-1/2} =
      \bB_0^{-1/2} \hat\bS(\bU)^2\bB_0^{-1/2} \succ\bzero\, .
\end{equation}
Meanwhile, since $\hat\bS(\bU) \in \cD_0$ by Lemma~\ref{lemma:image_Sstar_D_general}, we have,
as $y\to 0$,
\begin{equation}
\bB_y^{-1/2}\bS_y \E[(\bI + \bT \bS_y)^{-1}\bT \bB_y \bT (\bI + \bS_y^* \bT)^{-1}] \bS_y^*\bB_y^{-1/2} \to
\frac{1}{\alpha} \bB_0^{-1/2}
\hat\bS^{1/2}\cuT_{\hat\bS(\bU)}\big(\hat\bS^{-1/2}\bB_0\hat\bS^{-1/2}\big) \hat\bS^{1/2}\bB_0^{-1/2}\, ,
\end{equation}
where $\hat\bS \equiv \hat\bS(\bU).$
Using the last two displays in Eq.~\eqref{eq:Rearranged}, we obtain that
\begin{equation}
    \cuT_{\hat\bS(\bU)}(\bB')  \prec \bB'
\end{equation}
for the matrix (strictly) positive definite matrix $\bB' := \hat\bS^{-1/2}\bB_0 \hat\bS^{-1/2}$, so that $\hat\bS \in \tilde\cP_-$ of Lemma~\ref{lemma:set_characterization}, which is equal to $\cP_-.$
\end{proof}

\paragraph{Global injectivity of $\bK$ and connectedness of $\cP_-$.}
Note that non-degeneracy of the derivative of $\bK$ on $\cP_-$  guarantees that it is locally injective.
We will next show that $\cP_-$ is an open connected set and that $\bK$ is globally injective on it.
This will allow us to conclude that 
we cannot have $\hat \bS(\cU) \subsetneq \cP_-$ 
unless $\hat \bS$ has an analytic continuation through a boundary point of $\cU_-.$

We start off with the global injectivity in the following lemma.
\begin{lemma}[Injectivity of $\bK$ on $\cP_-$]
\label{lemma:K_injective}
    Assume $\nu$ is compactly supported. 
    Then $\bK$ is \emph{globally injective} on $\cP_-$.
\end{lemma}
\begin{proof}
Once again, we use the expectation notation. 
Fix $\bS,\tilde\bS\in\cP_-$
and define the operator 
\begin{equation}
   \cuG_{0}[\bDelta] := \alpha \E[\bZ(\bI+\bZ)^{-1}\bDelta \tilde\bZ(\bI+\tilde\bZ)^{-1}]
   \quad\quad\textrm{where}\quad\quad
    \bZ := \bS^{1/2} \bT\bS^{1/2},\;\tilde\bZ :=\tilde\bS^{1/2} \bT\tilde\bS^{1/2},
\end{equation}
for $\bDelta\in\sfM_k(\C).$

A straightforward calculation shows that this relates to the difference $\bK(\bS) -\bK(\tilde\bS)$ by
\begin{align}
    \alpha \bS^{1/2} \left(\bK(\bS) - \bK(\tilde\bS)\right)\tilde\bS^{1/2} = 
   \left(\bI - \cuG_0\right)[\bDelta_0],\quad\quad \bDelta_0 := \bS^{-1/2}(\bS -\tilde\bS)\tilde\bS^{-1/2}.
\end{align}
Hence, to prove injectivity, since $\bS,\tilde\bS\succ\bzero,$
it is sufficient to show that $\bI -\cuG_0$ is invertible. 
To do so, we will follow an argument similar to the one in Lemma~\ref{lemma:set_characterization} and bound the norm of compositions of $\cuG_0^p.$

Let $\bT_1,\dots,\bT_{2p}$ be distributed with respect to the product $\nu^{\otimes 2p}$.
Set for $i\in[2p],$
\begin{equation}
    \bA_{i} := \bZ_i (\bI+\bZ_i)^{-1},\quad\bZ_i :=  \bS^{1/2}\bT_i\bS^{1/2},\quad\quad\textrm{and}\quad\quad
     \tilde\bA_{i} := \tilde\bZ_i (\bI+\tilde\bZ_i)^{-1},\quad\tilde\bZ_i :=  \tilde\bS^{1/2}\bT_i\tilde\bS^{1/2}.
\end{equation}
Then for $\bDelta\in\sfM_{k}(\C)$, we have
\begin{align} 
\|\cuG_0^p(\bDelta)\|_\fro^2 
&= \alpha^{2p}\E[\Tr( 
\tilde\bA_{p} \dots \tilde\bA_1
\bDelta^*
\bA_1\dots\bA_p
\bA_{p+1} \dots \bA_{2p} 
\bDelta
\tilde\bA_{2p} 
\dots
\tilde\bA_{p+1}
)]\\
&= \alpha^{2p}\E[\Tr(
\bA_1 \dots \bA_{2p} 
\bDelta
\tilde\bA_{2p} 
\dots
\tilde\bA_1
\bDelta^*
)]\\
&\le \alpha^{2p}
{\underbrace{\E[\Tr( 
\bDelta^* \bA_{2p}\dots\bA_1 \bA_1\dots \bA_{2p} \bDelta
)]}_{T_1}}^{1/2}
{\underbrace{\E[\Tr( 
\bDelta \tilde\bA_1\dots\tilde\bA_{2p}\tilde\bA_{2p} \dots \tilde\bA_1\bDelta^{*}
)]}_{T_2}}^{1/2}.
\end{align}

Now since $\bS,\tilde\bS \in\cP_-$ which is equal to $\tilde\cP_-$ by Lemma~\ref{lemma:set_characterization}, there exist $\bB,\tilde\bB\succ\bzero$ so that
$\cuT_{\bS}[\bB] \prec \bB, \cuT_{\tilde\bS}[\tilde\bB]\prec \tilde\bB$.
Then we bound $T_1$ as 
\begin{align}
    T_1 &\le \|\bB^{-1}\|_\op  \E[ \Tr(\bDelta^{*}\bA_{2p} \dots  
    \bA_{2}\bB^{1/2}\E[\bB^{-1/2}\bA_{1}\bB\bA_1 \bB^{-1/2}]  \bB^{1/2}\bA_2 \dots \bA_{2p}\bDelta) ]\\
    &\le
    \|\bB^{-1}\|_\op
    \|\E[\bB^{-1/2}\bA_{1}\bB\bA_1 \bB^{-1/2}]\|_\op
    \E[ \Tr(\bDelta^{*}\bA_{2p} \dots  \bA_2\bB\bA_2 \dots \bA_{2p}\bDelta)]
\end{align}
which after iterating yields
\begin{equation}
    T_1 \le \alpha^{-p}\|\bB^{-1}\|_\op 
    \|\bB\|_\op  \|\bB^{-1/2}\cuT_\bS(\bB)\bB^{-1/2}\|_\op^p \Tr(\bDelta^{*}\bDelta)
\end{equation}
and a similar computation shows 
\begin{equation}
    T_2 \le \alpha^{-p}\|\tilde\bB^{-1}\|_\op 
    \|\tilde\bB\|_\op  \|\tilde\bB^{-1/2}\cuT_\bS(\tilde\bB)\tilde\bB^{-1/2}\|_\op^p \Tr(\bDelta^{*}\bDelta),
\end{equation}
so that for some $C>0$ and $\delta\in(0,1),$
\begin{equation}
    \|\cuG_0^{p}\|_{\fro\to\fro} \le 
    C  \delta^p.
\end{equation}
This shows that the Nuemann series $\sum_{p=0}^\infty  \cuG_{0}^p$ converges in $\|\,\cdot\,\|_{\fro\to\fro}$ norm, and thereby $\bI -\cuG_0$ is invertible
\end{proof}

Now we move on to showing that the connectedness of $\cP_-.$
\begin{lemma}
\label{lemma:connectedness}
Assume $\nu$ is compactly supported.
    Then the set $\cP_-$ is a connected open set.
\end{lemma}
\begin{proof}
Recall first the characterization $\tilde\cP_-$ from
Lemma~\ref{lemma:set_characterization}.
Observe that for a given $\bS \in\cD_0$, the condition
    $\cuT_\bS(\bB) \prec\bB$  
is equivalent to 
\begin{equation}
\alpha\E[\bT(\bS^{-1} +\bT)^{-1} \bB_0 (\bS^{-1} + \bT)^{-1} \bT ] \prec \bB_0 \quad\quad \bB_0 := \bS^{-1/2} \bB \bS^{-1/2}.
\end{equation}
So it is sufficient to show the connectedness of the set
\begin{equation}
    \label{eq:alt_connectedness_set}
    \cG := \{\bY \succ\bzero : \bY + \bT \succ\bzero \; \;\forall\bT\in\supp(\nu), \;\; 
    \exists \bB_0 \succ\bzero \;\textrm{s.t.}\;
  \cuF_0(\bY)[\bB_0]
  \prec\bB_0
    \}
\end{equation}
where 
\begin{equation}
 \cuF_0(\bY)[\bB_0] :=   
  \alpha\E[\bT(\bY +\bT)^{-1} \bB_0 (\bY + \bT)^{-1}\bT].
\end{equation}

First, fix
$\bY\in\cG$ and let $\bB_0\succ\bzero$ satsify $\cuF_0(\bY)[\bB_0] \prec \bB_0$. 
We will show that for all $t\ge0,$ the path
\begin{equation}
    \bY_t := \bY + t \bB_0 
\end{equation}
remains in $\cG.$
It is clear that $\bY_t + \bT \succ\bzero$ for $\bT\in\supp(\nu)$.
For the remaining condition, set
\begin{equation}
\bQ \equiv \bQ(\bT) := \bB_0^{-1/2} (\bY + \bT) \bB_0^{-1/2}.
\end{equation}
Since $\bY\in\cG$, $\bQ(\bT) \succ\bzero$ for $\bT\in\supp(\nu)$ and so $(\bQ + t \bI)^2 = \bQ^2 + 2t\bQ + t^2 \bI \succ \bQ^2$ for all $t\ge 0$. Recalling that the inverse is matrix monotone on PSD matrices we conclude that for all
\begin{equation}
\label{eq:ordering_key}
\bT\bB_0^{-1/2}(\bQ(\bT) + t\bI)^{-2}
\bB_0^{-1/2}\bT
\preceq 
\bT\bB_0^{-1/2}
\bQ(\bT)^{-2}
\bB_0^{-1/2}\bT
,
\end{equation}
almost-surely. 
Now observe that for all $t\ge0$,
\begin{align}
\E[\bT\bB_0^{-1/2}(\bQ(\bT) +t\bI)^{-2} \bB_0^{-1/2}\bT]
&= 
\E[\bT (\bY +\bT + t\bB_0)^{-1} \bB_0
(\bY +\bT + t\bB_0)^{-1}\bT]\\
&=
\E[\bT (\bY_t + \bT )^{-1} \bB_0
(\bY_t +\bT )^{-1}\bT]\\
&=\alpha^{-1}\cuF_0(\bY_t)[\bB_0],
\end{align}
 Eq.~\eqref{eq:ordering_key} then allows us to conclude that
    $\cuF_0(\bY_t)[\bB_0] \preceq \cuF_0(\bY)[\bB_0] \prec \bB_0$
so that $\bY_t \in \cG$.

Now let $C_\alpha> 0$ be chosen sufficiently large so that
\begin{eqnarray}
\alpha \E[\bT^2]     \le C_\alpha^2.
\end{eqnarray}
Consider the set
\begin{equation}
   \cG_0 := \{\bY \succ\bzero : \bY + \bT \succ C_\alpha \bI \;\; \forall\bT\in\supp(\nu)\}.
\end{equation}
Clearly, if $\bY \in\cG_0$, then $\bY + \bT \succ\bzero$ for $\bT\in\supp(\nu)$.
Further, for all $\bY \in\cG_0$ and $\bT\in\supp(\nu)$, 
we have
\begin{eqnarray}
    \cuF_0(\bY)[\bI] = 
    \alpha \E\big[\bT (\bY + \bT)^{-2} \bT\big] 
    \prec  \frac{\alpha}{C_\alpha^2} \E[\bT^2] \bI \preceq\bI,
\end{eqnarray}
so that $\bT \in\cG$.  Hence, $\cG_0$ is a convex subset of $\cG$.
Now if $\bY\in\cG$, with $\bB_0\succ \bzero$ 
satisfying $\cuF_0(\bY)[\bB_0] \prec \bB_0$, 
since $\bT$ is compactly supported, there exists $t>0$ large enough so that $\bY + t\bB_0 \in\cG_0$.

%
\end{proof}
\begin{proof}[Proof of Theorem~\ref{thm:K_transform_real}]
Let $\partial\cU_-$ denote the boundary
\begin{equation}
  \partial\cU_- = \big\{\bU \in\sa{k}{\C} :   
    w_k - (\bU \otimes\id)\succeq 0,\; 0\in\spec(w_k - (\bU\otimes \id))
    \big\}.
\end{equation}

Assume by contradiction that $\cU_- \neq \bK(\cP_-)$. Since 
$\hat\bS(\cU_-)\subseteq\cP_-$ by Lemma~\ref{lemma:image_U_in_P}, 
and $\bK(\hat\bS(\bU))=\bU$ for all $\bU\in\cU_-$ by Lemma \ref{lemma:regularity_free_ST}
it follows that $\cU_-\subseteq\bK(\cP_-)$.

The set $\cP_-$ is connected by Lemma~\ref{lemma:connectedness}, and hence so is $\bK(\cP_-)$.
Furthermore, 
both $\cU_-$ and $\bK(\cP_-)$ are open; the latter is open because $\cP_-$ is open and 
$\bK$ is a local diffeomorphism. As a result, we conclude that there must exist a point $\bU_0\in\bK(\cP_-)\cap\partial\cU_-$. 

%

Now by Lemma~\ref{lemma:K_injective}, (and the fact that $\cuD_\bS \bK(\bS)[\,\cdot\,]\succ\bzero$
for all $\bS\in\cP_-$), $\bK$ has an analytic inverse 
$\bK^{-1}$ on $\bK(\cP_-)$ and by Lemma~\ref{lemma:image_U_in_P} we must have
\begin{equation}
\label{eq:S_is_K_inv}
    \bU\in\cU_- \;\;\Rightarrow\;\; \hat\bS(\bU) = \bK^{-1}(\bU).
\end{equation}
Therefore, $\hat\bS$ can be continued analytically in a neighborhood of $\bU_0$.
But then the Stieltjes transform
\begin{equation}
   s_0(x) :=  \frac1{\alpha k}\left(\Tr\otimes\tau\right) \left(w_k - ( (\bU + x\bI )\otimes \id)\right)^{-1},\quad\quad
\bU_x := \bU + x\bI,\quad\quad x < 0
\end{equation}
would have an analytic continuation through $0$, which is a point in the support of the measure defined by $w_k -(\bU \otimes \id)$ through tracial moments by Proposition~\ref{prop:spectra_match}; a contradiction.

Finally, the equivalence of the definitions in Eq.~\eqref{eq:P_def} and Eq.~\eqref{eq:P_def_with_B}  
follows from Lemma~\ref{lemma:set_characterization} and 
Eq.~\eqref{eq:DK_to_cuT}.
\end{proof}

\subsection*{Statement of AI use}

The proof of Lemma~\ref{lemma:connectedness} was given 
by GPT 5.5 Pro, prompted by Alexandru Lopotenco with
the parameterization of $\cP_-$ appearing in 
Eq.~\eqref{eq:alt_connectedness_set}. We are grateful to the 
Alexandru for this.

\subsection*{Acknowledgements}

We are grateful to Ramon van Handel for his comments on this manuscript, and 
clarifications regarding their \cite{parmaksiz2025computing}.
This work was partially supported by the NSF through
Award MFAI-2501597.

\bibliographystyle{amsalpha}

\providecommand{\bysame}{\leavevmode\hbox to3em{\hrulefill}\thinspace}
\providecommand{\MR}{\relax\ifhmode\unskip\space\fi MR }
\providecommand{\MRhref}[2]{%
  \href{http://www.ams.org/mathscinet-getitem?mr=#1}{#2}
}
\providecommand{\href}[2]{#2}

\newpage
\appendix
\section{Proofs of Lemma~\ref{lemma:inv_on_upper_plane}}
\label{app:proofs_FP}

In this appendix, we will give the proof of
Lemma~\ref{lemma:inv_on_upper_plane} (uniqueness and existence of the solution to 
the matrix fixed-point equation $\bK_{\alpha,\nu}(\bS)=\bZ$ on $\bbH_+^k$).

Our proof will be based on a similar argument that appeared in~\cite{asgari2025local}, but for $\bZ = z\bI_k$.
The argument is similar to the one used to prove Lemma~\ref{lemma:K_injective} .

Throughout the appendix we fix 
$(\alpha,\nu, k)$ as in Lemma~\ref{lemma:inv_on_upper_plane}.
\subsection{Preliminary identities}
\label{app:notation_aux}
We begin with some preliminary identities.

\begin{lemma}
\label{lemma:Im_Sinv}
  For $\bS\in\bbH_+^k$ and $\bT\in \sa{k}{\C}$,
     $\bS$ is invertible and satisfies
    \begin{equation}
\label{eq:Im_Sinv}
    \|\bS^{-1}\|_\op \le \|\Im(\bS)^{-1}\|_\op,\quad\quad \Im(\bS^{-1}) = -\bS^{-1}\Im(\bS)\bS^{*-1}\prec\bzero.
    \end{equation}
\end{lemma}
\begin{proof}
The identity  on the left of Eq.~\eqref{eq:Im_Sinv} can be seen by writing
\begin{equation}
    \Im(\bS^{-1}) 
     = \frac{1}{2i} (\bS^{-1}-\bS^{*-1})  
     = \bS^{-1}\Big(\frac{1}{2i} (\bS^{*} - \bS)\Big) \bS^{*-1} = -\bS^{-1}\bB\bS^{*-1}.
\end{equation}

For the operator-norm bound, let $\bA = \Re(\bS)$ and $\bB = \Im(\bS)$. Then we have
\begin{equation*}
\bS = \bA + i\bB
= \bB^{1/2}\bigl(\underbrace{\,\bB^{-1/2}\bA\bB^{-1/2}\,}_{=: \,\bH} + i\bI\bigr)\bB^{1/2}.
\end{equation*}
Since $\bH$ is Hermitian with real spectrum, $\|(\bH+i\bI)^{-1}\|_\op \le 1$, hence 
$\|\bS^{-1}\|_\op \le \|\bB^{-1/2}\|_\op^2 \|(\bH+i\bI)^{-1}\|_\op \le \|\bB^{-1}\|_\op$.
\end{proof}

As a consequence of the above lemma,  we have 
for any $\bS \in \bbH_+^k$ and $\bT\in \sa{k}{\C}$ that 
 $\bI + \bS\bT$ is invertible: 
indeed, $\bI + \bS\bT = \bS(\bS^{-1} + \bT)$, and 
$\Im(\bS^{-1}+\bT) = \Im(\bS^{-1}) = -\bS^{-1}\bB\bS^{*-1}\prec\bzero.$
For such $(\bS,\bT)$, set 
\begin{equation}
\label{eq:bfeta_def_app}
\bfeta(\bS,\bT) := \bT(\bI_k + \bS\bT)^{-1} = (\bI_k + \bT\bS)^{-1}\bT.
\end{equation}
In this notation, the $K$-transform~\eqref{eq:K_transform} and 
the auxiliary matrix $\bF_\bZ(\bS)$ reads
$\bK_{\up{\alpha,\nu}}(\bS) = \int \bfeta(\bS,\bT)\, \nu(\de\bT) - \alpha^{-1}\bS^{-1}.$

We will repeatedly use the following identities. 
\begin{lemma}\label{lemma:eta_identities_app}
Fix $\bS \in \bbH_+^k$, $\bT \in \sa{k}{\C}$ and $\bZ \in \bbH_+^k$, and let 
$\bB := \Im(\bS)$, $\bfeta := \bfeta(\bS,\bT)$. 
Then  the following hold:
\begin{enumerate}
\item $\bfeta$ satisfies
\begin{equation}
\label{eq:Im_eta}
\Im(\bfeta) = -\bfeta\, \bB\, \bfeta^{*},\qquad \|\bfeta\|_\op \le \|\bB^{-1}\|_\op\, .
\end{equation}
\item
If $\bS$ satisfies $\bK_{\up{\alpha,\nu}}(\bS) = \bZ$, then
\begin{align}
\label{eq:imag_part_FP}
&\int \bfeta\,\bB\,\bfeta^*\, \nu(\de\bT) 
= \frac{1}{\alpha}\bS^{-1}\bB\bS^{*-1} - \Im(\bZ),\\
\label{eq:key_identity}
&\bI = \alpha^2\, \bS\int \bfeta\,\bB\,\bfeta^*\, \nu(\de\bT)\, \bS^*\bB^{-1} 
+ \alpha\, \bS\,\Im(\bZ)\,\bS^* \bB^{-1}.
\end{align}
\end{enumerate}
\end{lemma}
\begin{proof}

To show Item \textit{1},
note that $\bfeta(\bS,\bT)^* = (\bI + \bT\bS^*)^{-1}\bT$, hence
\begin{align*}
\bfeta - \bfeta^*
&= (\bI + \bT\bS^*)^{-1}\bigl[(\bI+\bT\bS^*)\bT - \bT(\bI+\bS\bT)\bigr](\bI+\bS\bT)^{-1} \\
&= (\bI+\bT\bS^*)^{-1}\bT(\bS^*-\bS)\bT(\bI+\bS\bT)^{-1}
= -2i\, \bfeta^*\bB\bfeta.
\end{align*}
Dividing by $2i$ yields the first relation in Eq.~\eqref{eq:Im_eta}.
For the second relation, note that if $\bT$ is invertible, then we have
$\bfeta = (\bT^{-1}+\bS)^{-1}$ and the second identity of Eq.~\eqref{eq:Im_eta} follows 
by the bound in Lemma~\ref{lemma:Im_Sinv}. 
Meanwhile, the general non-invertible case follows by approximation.

To show Item \textit{2},
take imaginary parts in
\begin{equation*}
\bZ = \int\bfeta(\bS,\bT)\,\nu(\de\bT) - \frac{1}{\alpha}\,\bS^{-1}
\end{equation*}
and use Eq.~\eqref{eq:Im_Sinv} and Eq.~\eqref{eq:Im_eta} to get Eq.~\eqref{eq:imag_part_FP}.
Multiply~\eqref{eq:imag_part_FP} by $\alpha\bS$ on the left and $\bS^*\bB^{-1}$ on the right, then rearrange to obtain~\eqref{eq:key_identity}.
\end{proof}

\subsection{Uniqueness via convergence of a Neumann series}
\label{app:uniqueness}
Uniqueness of the solution $\bS\in\bbH_+^k$ to $\bK_{\up{\alpha,\nu}}(\bS) = \bZ$
will be proved by
an approach similar to the one used to prove Lemma~\ref{lemma:K_injective}. 

For convenience, we define for $\bZ,\bS\in\bbH_+^k$ 
\begin{align}
\label{eq:bF_def}
\bF_\bZ(\bS) &:= \Big(\int \bfeta(\bS,\bT)\, \nu(\de\bT) - \bZ\Big)^{-1}.
\end{align}
Note the matrix fixed-point equation $\bK_{\up{\alpha,\nu}}(\bS)=\bZ$ is 
equivalent to $\bF_\bZ(\bS) = \alpha\bS$.

Given $\bS, \tilde\bS,\bZ\in\bbH_+^k$, 
$\bF_\bZ(\bS)$ and $\bF_\bZ(\tilde\bS)$ are well-defined since by Lemma~\ref{lemma:eta_identities_app},
$\Im(\bfeta - \bZ ) \prec\bzero$.
 We define the (linear) operator
$\cuL_{\bS,\tilde\bS}:\C^{k\times k}\to \C^{k\times k}$ by
\begin{equation}
\label{eq:T_op_def}
\cuL_{\bS,\tilde\bS}[\bDelta] := \bF_\bZ(\bS) \int \bfeta(\bS, \bT)\, \bDelta\, \bfeta(\tilde\bS, \bT)\, \nu(\de\bT)\, \bF_\bZ(\tilde\bS),
\end{equation}
and write $\cuL_{\bS} := \cuL_{\bS,\bS}$.
Note that we suppress the dependence on $\bZ$ in the notation for convenience.

This operator satisfies the following consequence of the resolvent identity:
\begin{equation}
\label{eq:F_diff}
\bF_\bZ(\bS) - \bF_\bZ(\tilde\bS) = \cuL_{\bS,\tilde\bS}[\bS - \tilde\bS]\, .
\end{equation}
In particular, if $\bS$ and $\tilde\bS$ are both solutions 
of $\bK_{\up{\alpha,\nu}}(\bS) = \bZ$ in $\bbH_+^k$, then 

\begin{equation}
\label{eq:two_sols}
(\alpha\, \id - \cuL_{\bS,\tilde\bS})[\bS -\tilde\bS] = \bzero.
\end{equation}
Hence, uniqueness of the solution in $\bbH_+^k$ will follow 
from the invertibility of $(\alpha\,\id - \cuL_{\bS,\tilde\bS})$ on $\C^{k\times k}$, 
which in turn follows from convergence of the Neumann series 
$\sum_{q\ge 0}\alpha^{-q-1}\cuL_{\bS,\tilde\bS}^q$. The next lemma provides a bound on the 
operator norm of $\cuL_{\bS,\tilde\bS}^q$ in terms of a generic positive-definite matrix
$\bB$; specializing $\bB = \Im(\bS)$, $\tilde\bB = \Im(\tilde\bS)$ will give the bound 
at fixed-point solutions necessary to prove convergence of the Neumann series.

\begin{lemma}[Operator-norm bound via Cauchy--Schwarz]\label{lemma:T_op_norm_bound}
Fix $\bZ\in\bbH_+^k$ and $\bS,\tilde\bS\in\bbH_+^k$ such that $\bF_\bZ(\bS),\bF_\bZ(\tilde\bS)$ 
are well-defined, and let $\bB,\tilde\bB\succ\bzero$.
For all integers $q\ge 1$,
\begin{equation}
\label{eq:T_op_norm_bound}
\|\cuL_{\bS,\tilde\bS}^q\|_{\op\to\op}^2 
\le \|\bB\|_\op\|\bB^{-1}\|_\op\|\tilde\bB\|_\op\|\tilde\bB^{-1}\|_\op\cdot \Gamma_\bS(\bB)^q\,\Gamma_{\tilde\bS}(\tilde\bB)^q, 
\end{equation}
where 
\begin{equation}
\label{eq:Gamma_def}
\Gamma_\bS(\bB) := \Big\|\bB^{-1/2}\bF_\bZ(\bS)\int \bfeta(\bS,\bT)\,\bB\,\bfeta(\bS,\bT)^*\nu(\de\bT)\, \bF_\bZ(\bS)^* \bB^{-1/2}\Big\|_\op.
\end{equation}
Here $\|\cdot\|_{\op\to\op}$ is the operator norm on linear maps 
$(\C^{k\times k},\|\cdot\|_\op) \to (\C^{k\times k},\|\cdot\|_\op)$.
\end{lemma}

\begin{proof}
Let $\bT_1,\dots,\bT_q,\tilde\bT_1,\dots,\tilde\bT_q$ be i.i.d.\ copies of $\bT\sim\nu$, 
and write $\bfeta_i := \bfeta(\bS,\bT_i)$, 
$\tilde\bfeta_i := \bfeta(\tilde\bS,\bT_i)$. By independence and the definition of $\cuL_{\bS,\tilde\bS}$, 
\begin{align}
\cuL_{\bS,\tilde\bS}^q[\bDelta]
 = \E\bigg[\prod_{i=q}^1\big(\bF_\bZ(\bS)\bfeta_i\big)\,\bDelta\,\prod_{i=1}^q\big(\tilde\bfeta_i\bF_\bZ(\tilde\bS)\big)\bigg]\, .
\end{align}
Fix unit vectors $\bu, \bv \in\C^k$ and $\bDelta\in\C^{k\times k}$. 
We have
by an argument similar to the one in Lemma~\ref{lemma:K_injective},
\begin{equation}
|\bv^*\cuL_{\bS,\tilde\bS}^q[\bDelta]\bu|^2 \le \mathrm{(I)}^2\cdot\mathrm{(II)}^2,
\end{equation}
where
\begin{align}
\mathrm{(I)} &= \E\bigg[\Tr\bigg(
\Big|
\bDelta^* \prod_{i=1}^q(\bfeta_i^*\bF_\bZ(\bS)^*)\,\bv\bv^*\,\prod_{i=q}^1(\bF_\bZ(\bS)\tilde\bfeta_i)
\Big|^2 \bigg)\bigg],\\
\mathrm{(II)} &= \E\bigg[\Tr\bigg(
\Big|
\bDelta \prod_{i=1}^q(\tilde\bfeta_i\bF_\bZ(\tilde\bS))\,\bu\bu^*\,\prod_{i=q}^1(\bF_\bZ(\tilde\bS)^*\bfeta_i^*)
\Big|^2 \bigg)\bigg].
\end{align}
We bound (II); the argument for (I) is identical.

By the cyclicity of the trace, 
\begin{align}
\mathrm{(II)} &= 
\E\bigg[\bu^*\prod_{i=q}^1\big(\bF_\bZ(\tilde\bS)^*\tilde\bfeta_i^*\big)\bDelta^*\bDelta
\prod_{i=1}^q\big(\tilde\bfeta_i\bF_\bZ(\tilde\bS)\big)\bu\bigg] 
\cdot
\E\bigg[
\bu^*\prod_{i=q}^1\big(\bF_\bZ(\tilde\bS)^*\bfeta_i^*\big)\,
\prod_{i=1}^q\big(\bF_\bZ(\tilde\bS)\bfeta_i\big)\,
\bu\bigg].
\end{align}
We now
iteratively use that for PSD matrix $\bH$, we have $\bw^*\bH\bw \le \|\bH\|_\op \bw^*\bw$, and
 insert $\bB^{1/2}\bB^{-1}\bB^{1/2} = \bI$ into the expression to extract the desired 
powers of the term $\Gamma_{\tilde\bS}(\tilde\bB).$
 Iterating $q$ for each of the two expectations in $\mathrm{(II)}$ yields
\begin{equation}
\mathrm{(II)} \le
\Gamma_{\tilde\bS}(\tilde\bB)^{2q}
\|\tilde\bB^{-1}\|_\op^2 \|\tilde\bB\|_\op^2\|\bDelta\|_\op^2,
\end{equation}
The same argument bounds $\mathrm{(I)} \le \Gamma_{\bS}(\bB)^{2q}\|\bB^{-1}\|_\op^2\|\bB\|_\op^2\|\bDelta\|_\op^2$, 
and combining the two bounds and taking supremum over $\|\bDelta\|_\op\le 1$ proves Eq.~\eqref{eq:T_op_norm_bound}.
\end{proof}

The next lemma controls $\Gamma_\bS(\bB)$ at fixed-point solutions. 
The key observation is that for $\bS$ solving $\bF_\bZ(\bS) = \alpha\bS$, the matrix
inside the operator norm in Eq.~\eqref{eq:Gamma_def} has an explicit form via 
the imaginary-part identity~\eqref{eq:imag_part_FP}.

\begin{lemma}[Operator-norm bound at fixed-point solutions]\label{lemma:Gamma_at_FP}
Fix $\bZ\in\bbH_+^k$ and let $\bS\in\bbH_+^k$ be a solution of $\bK_{\alpha,\nu}(\bS) = \bZ$. 
With $\bB := \Im(\bS)$, we have the bound
\begin{equation}
\label{eq:Gamma_FP_bound}
\Gamma_\bS(\bB) \le \alpha- \alpha^2\lambda_{\min}\big(\bB^{-1/2}\bS\,\Im(\bZ)\, \bS^*\bB^{-1/2}\big)\, .
\end{equation}
In particular, 
$$\frac{1}{\alpha}\Gamma_\bS(\bB) < 1$$. 
\end{lemma}

\begin{proof}
By the fixed-point equation $\bF_\bZ(\bS) = \alpha\bS$, the matrix inside the operator 
norm in Eq.~\eqref{eq:Gamma_def} is
\begin{equation}
\bU := \alpha^2\bB^{-1/2}\bS\int\bfeta(\bS,\bT)\bB\bfeta(\bS,\bT)^*\,\nu(\de\bT)\,\bS^*\bB^{-1/2}\, .
\end{equation}
On the one hand, $\bU\succeq \bzero$.
 On the other hand, conjugating the 
imaginary-part identity Eq.~\eqref{eq:imag_part_FP} by $\alpha\bS$  then by $B^{-1/2}$,
\begin{equation}
\bU = \alpha\bI -\alpha^2 \bH, \quad\quad \bH := \bB^{-1/2}\bS\,\Im(\bZ)\,\bS^*\bB^{-1/2}\, .
\end{equation}
Since $\Im(\bZ)\succ\bzero$ and $\bS$ is invertible, the claim follows.
\end{proof}

\begin{proof}[Proof of Lemma~\ref{lemma:inv_on_upper_plane}]
That $\hat S_{\alpha,\nu}(\bZ)$ of Eq.~\eqref{eq:free_ST_def} satisfies 
the matrix fixed-point equation $\bK_{\alpha,\nu}(\bS)=\bZ$ on $\bbH_+^k$ can be verified via several established approaches 
(see remark~\ref{rmk:PropoProof}).
For example, following the approach of~\cite{asgari2025local}, one can show that 
the empirical Stieltjes transform $\hat\bS_n(\bZ) = n^{-1}(\bI_k\otimes \Tr)\bR(\bZ)$ of Eq.~\eqref{eq:empirical_ST}
satisfies the fixed-point equation in an approximate sense, meaning that $\bK_{\alpha,\nu}(\hat\bS_n(\bZ)) =  \bZ + \bE_n$
for some error term whose norm $\|\bE_n\|_\op$ converges to 0 in probability as $n\to\infty$, 
and then noting that $\hat\bS_{\alpha,\nu}(\bZ)$ is the probability limit of $\hat\bS_n(\bZ)$ allows one to conclude 
that $\hat\bS_{\alpha,\nu}(\bZ)$ satisfies the fixed-point equation.

For uniqueness, let $\bS,\tilde\bS\in\bbH_+^k$ be two solutions. 
Then Lemma~\ref{lemma:T_op_norm_bound} applied with $\bB = \Im(\bS), \tilde\bB = \Im(\tilde\bS)$ 
yields
\begin{equation}
\|(\alpha^{-1}\cuL_{\bS,\tilde\bS})^q\|_{\op\to\op}^{2} 
\le \|\bB\|_\op\|\bB^{-1}\|_\op \|\tilde\bB\|_\op\|\tilde\bB^{-1}\|_\op \cdot \big(\alpha^{-1}\Gamma_\bS(\bB)\big)^q\big(\alpha^{-1}\Gamma_{\tilde\bS}(\tilde\bB)\big)^q\, .
\end{equation}
Since $\alpha^{-1}\Gamma_\bS(\bB),\alpha^{-1}\Gamma_{\tilde\bS}(\tilde\bB) <1$ by 
Lemma~\ref{lemma:Gamma_at_FP}, 
the
Neumann series $\sum_{q\ge 0}\alpha^{-q}\cuL_{\bS,\tilde\bS}^q$
converges absolutely, 
so that $\id - \alpha^{-1}\cuL_{\bS,\tilde\bS}$ is invertible.
Recalling Eq.~\eqref{eq:two_sols} asserting 
$(\alpha\,\id - \cuL_{\bS,\tilde\bS})[\bS - \tilde\bS] = \bzero$, 
we conclude that $\bS = \tilde\bS$.
\end{proof}

\end{document}